\documentclass[journal]{IEEEtran}
\PassOptionsToPackage{bookmarks=false}{hyperref}
\usepackage{amsmath, amsthm, amsfonts, amssymb, mathrsfs, float}
\usepackage{latexsym}
\usepackage{mathrsfs}
\usepackage[ruled,linesnumbered]{algorithm2e}
\usepackage{comment}
\usepackage{bm,bbm}
\usepackage{verbatim}
\usepackage{mathtools}
\usepackage{dsfont}
\usepackage{booktabs}
\usepackage{relsize}
\usepackage{caption}
\usepackage{subcaption}
\usepackage{color}
\usepackage[colorlinks,linkcolor=black,anchorcolor=black,citecolor=black,urlcolor=black]{hyperref}
\usepackage{cite}
\usepackage{multirow}
\theoremstyle{plain}  
\ifCLASSINFOpdf
\else
 \usepackage{graphicx} 
\usepackage{epstopdf}
\fi

\usepackage{url}
\usepackage{caption}
\usepackage{subcaption}
\usepackage{multirow}
\usepackage[shortlabels]{enumitem}
\usepackage{tikz}
\usepackage{pgfplots}
\usepackage{threeparttable} 
\usepackage{cleveref}
\usepackage{hyperref}
\allowdisplaybreaks[3] 

\def\mB{\mathcal{B}}

\def\mN{\mathcal{N}}
\def\mO{\mathcal{O}}
\def\mS{\mathcal{S}}

\newtheorem{lemma}{Lemma}
\newtheorem{thm}{Theorem}
\newtheorem{defi}{Definition}
\newtheorem{coro}{Corollary}

\newtheorem{fact}{Fact}

\def\bA{\bm{A}}
\def\bB{\bm{B}}
\def\bC{\bm{C}}
\def\bE{\bm{E}}

\def\bI{\bm{I}}

\def\bP{\bm{P}}
\def\bQ{\bm{Q}}

\def\bU{\bm{U}}
\def\bV{\bm{V}}
\def\bX{\bm{X}}

\def\bSigma{\bm{\Sigma}} 

\def\b0{\bm{0}}

\def\bv{\bm{v}}
\def\bx{\bm{x}}
\def\by{\bm{y}}

\newcommand{\z}{\mathbf 0}
\newcommand{\R}{\mathbb R}

\DeclareMathOperator{\sign}{sgn}
\DeclareMathOperator{\dist}{dist}

\DeclareMathOperator{\dom}{dom}
\DeclareMathOperator*{\argmin}{argmin}

\begin{document}
\title{A Linearly Convergent Algorithm for Rotationally Invariant $\ell_1$-Norm 
Principal Component Analysis}

\author{
	Taoli Zheng, Peng Wang, and Anthony Man-Cho So, \it{Senior Member, IEEE}\thanks{T. Zheng and A. M.-C. So are with the Department of Systems Engineering and Engineering Management, The Chinese University of Hong Kong, Hong Kong SAR, China (e-mails: {\tt \{tlzheng,manchoso\}@se.cuhk.edu.hk}). P. Wang is with  the Department of Electrical Engineering and Computer Science, University of Michigan, Ann Arbor (e-mail: \tt pengwa@umich.edu).} \\
}

\maketitle

\begin{abstract}
	To do dimensionality reduction on the datasets with outliers, the $\ell_1$-norm principal component analysis (L1-PCA)   as a typical robust alternative of the conventional PCA has enjoyed great popularity over the past years. In this work, we consider a rotationally invariant L1-PCA, which is hardly studied in the literature. To tackle it, we propose  a proximal alternating linearized minimization method with a nonlinear extrapolation for solving its two-block reformulation. Moreover, we show that the proposed method converges at least linearly to a limiting critical point of the reformulated problem. Such a point is proved to be a critical point of the original problem under a condition imposed on the step size. Finally, we conduct numerical experiments on both synthetic and real datasets to support our theoretical developments and demonstrate the efficacy of our approach. 
\end{abstract}


\IEEEpeerreviewmaketitle

\section{Introduction}\label{sec:intr}
With the widespread availability of high-dimensional data, principal component analysis (PCA) \cite{jolliffe2016principal} as a typical dimensionality reduction technique plays an increasingly important role in data analysis. It has found extensive applications in diverse fields such as computer vision \cite{yang2004two,moon2001computational,faruqe2009face}, image processing \cite{ke2004pca,bouwmans2018applications}, and network analysis \cite{lakhina2004diagnosing,lei2015consistency}, to name a few. Due to the fact that the conventional PCA (also known as L2-PCA) is sensitive to corruptions (generically referred to as {\it outliers} \cite{barnett1984outliers}) in the datasets, L1-PCA as a robust alternative of L2-PCA has recently received significant attention; see, e.g., \cite{devlin1981robust,lerman2018overview}. In this work, we consider a particular form of L1-PCA as follows: 
\begin{equation}\label{PCA:L1}
	\max\left\{\|\bQ \bQ^T\bX\|_1:\bQ \in {\rm St}(d,K)\right\}.
\end{equation} 
Here, $\bX = \left[ \bx_1,\dots,\bx_n\right] \in  \R^{d\times n}$ is the data matrix with the sample mean being zero, where $n$ and $d$ respectively denote the number of samples and dimension of the data points, $\|\bA\|_1=\sum_{i,j}|a_{ij}|$ denotes the $\ell_1$-norm of the matrix $\bA$, $K$ is the dimension of subspace with $K \le \min\{n,d\}$, and ${\rm St}(d,K)=\left\{\bQ \in \R^{d\times K}:\bQ^T \bQ=\bI_K\right\}$ denotes the Stiefel manifold. Generally, this problem can be used to find a low-dimensional subspace underlying a corrupted dataset by minimizing the variation of the projections of data points onto the subspace measured by the $\ell_1$-norm. Different from the formulation of Problem \eqref{PCA:L1}, 
there are actually many other formulations for L1-PCA  
that have been well studied in the literature; see, e.g., \cite{baccini1996l1,ding2006r,wang2021linear,nie2021non,park2016iteratively,ke2005robust}.
However, as pointed out in \cite{lerman2018overview}, no work focuses on Problem \eqref{PCA:L1}. To fill this gap, this work is devoted to develop an efficient algorithm for solving Problem \eqref{PCA:L1} with convergence analysis.  

Now, we review different formulations for L1-PCA in the literature. It is known that PCA admits two common interpretations for dimension reduction. One is that PCA aims to find a low-dimensional subspace for which the projections of data points onto it preserve the most variance. Specifically, the corresponding formulation known as L2-PCA is given by
\begin{equation}\label{PCA:max-L2}
	\max\left\{ \|\bQ\bQ^T \bX\|_F = \|\bQ^T \bX\|_F: \bQ \in {\rm St}(d,K) \right\},
\end{equation}
The other sheds light on another perspective of PCA that aims to find a low-dimensional subspace for which the projections of data points onto it best approximate these data points. This leads to a minimization formulation of the form 
\begin{equation}\label{PCA:min-L2}
	\min\left\{\|\bX -\bQ\bQ^T \bX\|_F: \bQ \in {\rm St}(d,K) \right\}. 
\end{equation} 
One can verify that this problem is equivalent to Problem \eqref{PCA:max-L2}. In particular, both of them admit a closed-form solution that can be solved efficiently by computing the singular value decomposition (SVD) of the data matrix $\bX$. Moreover, the subspace spanned by the columns of the obtained solution possesses many nice properties \cite{jolliffe2016principal}. Nevertheless, the above two formulations have an essential defect that they are sensitive to outliers in the data matrix due to the fact that the $\ell_2$-norm is not a robust scale function; see, e.g., \cite{devlin1981robust}. The outliers could significantly disturb the singular values of the data matrix and mislead the learning of the desired subspace. To tackle this problem, a recent line of research on robust subspace recovery becomes more and more popular \cite{lerman2018overview} and a great deal of work adopting different norms as scale functions have been done in the literature; see, e.g., \cite{brooks2013pure,ke2005robust,park2016iteratively,ding2006r,nie2021non}. Notably, taking the $\ell_1$ norm as the scale function gives L1-PCA. 
On one hand, along the line of maximizing the variance of projected data as in Problem \eqref{PCA:max-L2}, we can naturally obtain Problem \eqref{PCA:L1} and 
\begin{equation}
	\label{eq:max_l1}
	\max\left\{\|\bQ^T \bX\|_1: \bQ \in {\rm St}(d,K)\right\}.
\end{equation}
Compared to L2-PCA, the $\ell_1$-norm endows PCA with robustness but introduces non-smoothness into the objective function simultaneously. Consequently, Problems \eqref{PCA:L1} and \eqref{eq:max_l1} are both non-convex and non-smooth, and have no closed-form solutions. Unlike the equivalence of L2-PCA in Problem \eqref{PCA:max-L2}, Problems \eqref{PCA:L1} and \eqref{eq:max_l1} are generally not equivalent. Indeed, the objective of the latter problem is the pointwise maximum of linear functions, while the former is the pointwise maximum of quadratic functions. In particular, Problem \eqref{eq:max_l1} is widely studied in the literature. For example, Kwak \cite{kwak2008principal} proposed a fixed-point method for solving Problem \eqref{eq:max_l1} in a greedy manner and later Nie et al. \cite{nie2011robust} improved it to a non-greedy manner. Recently, Markopoulos et al. \cite{markopoulos2017efficient} developed an algorithm based on bit-flipping iteration. More recently, Wang et al. \cite{wang2019globally,wang2021linear} proposed a proximal alternating minimization method with extrapolation for solving it. We
refer the reader to \cite{markopoulos2018outlier} for more algorithms for solving Problem \eqref{eq:max_l1}. 
Besides, it is known that L2-PCA is rotationally invariant in the sense that the objective value remains unchanged under arbitrary subspace rotations. Specifically, the objective values of Problems \eqref{PCA:max-L2} and \eqref{PCA:min-L2} are unchanged if we replace $\bQ$ by $\bQ\bU$ for any orthogonal matrix $\bU \in \R^{K\times K}$. However, this property generally does not hold for the L1-PCA formulation \eqref{eq:max_l1}. Fortunately, the formulation \eqref{PCA:L1} keeps this property. This also motivates us to study Problem \eqref{PCA:L1}. In the following, we refer to it as {\it rotationally invariant} L1-PCA. On the other hand, along the line of minimizing the approximation error of data points as in Problem \eqref{PCA:min-L2}, we can obtain another L1-PCA formulation of the form
\begin{equation}\label{eq:l1_min}
	\min\left\{ \|\bX-\bQ \bQ^T\bX\|_1: \bQ \in {\rm St}(d,K) \right\}.
\end{equation}
This problem has been considered in the literature. For example, Baccini et al. \cite{baccini1996l1} proposed a heuristic approach based on a canonical correlation analysis. Recently, Park et al. \cite{park2016iteratively} proposed an iteratively reweighted least squares (IRLS) methods and provided convergence analysis. 
More related algorithms can be found in \cite{minnehan2018grassmann,liu2021robust}. 
Then, we review some other robust alternatives of L2-PCA closely related to L1-PCA. 
A popular one is the low-rank matrix factorization formulation of the form
\begin{equation}
	\min \left\{\|\bX-\bU\bV\|_1: \bU \in \R^{d\times K},\ \bV \in \R^{K\times n} \right\}.
\end{equation}
Ke and Kanade \cite{ke2005robust} proposed an alternating convex minimization method for solving it. Later, Yu et al. \cite{yu2012efficient} considered an augmented lagrange multiplier method and Eriksson and Hengel proposed a generalization of the Wiberg algorithm \cite{eriksson2012efficient} for tackling it. 
Up to now, the above formulations all utilize the $\ell_1$-norm as the scale function to achieve robustness. Actually, many other different norms have been employed in the literature. A notable one is the $\ell_{2,1}$-norm, which is defined as $\|\bA\|_{2,1}=\sum_{i}(\sum_{j}a_{ij}^2)^{1/2}$. This is also known as the $R_1$-norm in \cite{ding2006r}. 
This norm, together with the second interpretation of PCA, motivates the R1-PCA formulation as follows:   
\begin{equation}\label{eq:l21_min}
	\min \left\{\|\bX-\bQ\bQ^T\bX\|_{2,1}: \bQ \in {\rm St}(d,K) \right\}.
\end{equation}
It is noteworthy that this problem not only satisfies the mentioned subspace rotational invariance but also satisfy another {\it rotationally invariant} property in the sense that it is invariant to the change of coordinates in $\R^{d}$. 
Specifically, under data transformations $\bX \rightarrow \bU\bX$ for any orthogonal matrix $\bU \in \R^{d\times d}$, the objective value is unchanged when $\bQ \rightarrow \bU\bQ$. That is to say that when the sample space is rotated, the feature subspace will be rotated in the same way. 
For the sake of clarity, in our paper, we  refer to rotationally invariance as invariance with respect to \emph{subspace rotations}.
To solve this problem, Ding et al. \cite{ding2006r} proposed a subspace iteration algorithm. Later, Wang et al. \cite{wang2017ell} generalized this robust formulation via replacing $\ell_{2,1}$-norm with $\ell_{2,p}$-norm for any $p\in (0,2)$. Recently, Nie et al. \cite{nie2021non} considered a robust PCA model using the $\ell_{2,1}$-norm based on the first interpretation of PCA, i.e.,
\begin{equation}
	\max \left\{\|\bQ^T\bX\|_{2,1}:\bQ \in {\rm St}(d,K)\right\}.
\end{equation}
They developed an efficient non-greedy method to solve this problem. We refer the reader to \cite{lerman2018overview} for more formulations of robust PCA. 

\subsection{Contribution}

In this work, our first contribution is to propose a \emph{proximal alternating linearized minimization} (PALM) algorithm for solving a two-block reformulation (see Problem \eqref{PCA:L1-Re}) of Problem \eqref{PCA:L1}. Motivated by the encouraging performance of extrapolation techniques for accelerating non-convex optimization problems (see, e.g., \cite{wen2018proximal,Li2015,lu2019enhanced,wang2021linear}), we incorporate a quadratic extrapolation into the update of a block variable to possibly accelerate the PALM method. It is worth noting that compared to the widely used linear extrapolation in the literature (see, e.g., \cite{wen2018proximal,Li2015,lu2019enhanced,wang2021linear}),  our proposed quadratic  extrapolation (see \eqref{update:P-extra}) seems to be new. We refer to the resulting method as {\it proximal alternating linearized minimization with extrapolation} (PALMe). PALMe is computationally efficient since the update of each block of PALMe admits a closed-form solution and per-iteration cost is $\mO(ndK+dK^2)$.  
Our second contribution is to show that the proposed method converges at least linearly to a limiting critical point  (see Definition \ref{def:crit}) of the reformulation problem and this point is proved to be a critical point of Problem \eqref{PCA:L1} (see Definition \ref{def:crit-l}) under a verifiable condition imposed on the step size. To this end, we show that the Kurdyka-\L ojasiewicz (K\L) exponent (see Definition \ref{def:KL-exponent}) of Problems \eqref{PCA:L1-Re} is $1/2$. With this characterization, we employ the convergence analysis framework in \cite{bolte2014proximal} to show the linear convergence of PALMe. Our third contribution is to conduct numerical experiments on synthetic and real data sets to support our theoretical results. Specifically, the experimental results demonstrate that our approach is competitive, in terms of both numerical efficiency and clustering accuracy, with other PALM-type methods for solving Problem \eqref{PCA:L1}. We also compare our approach with various robust PCA approaches on image reconstruction, which illustrates the robustness to outliers of our approach.

Lastly, let us highlight the differences between this work and a closely related one \cite{wang2021linear}. First, the proximal alternating minimization method with extrapolation (PAMe)  in \cite{wang2021linear} is developed for solving Problem \eqref{eq:max_l1}, in which the objective is a pointwise maximum of linear functions. By contrast, the objective of our considered Problem \eqref{PCA:L1-Re} is a pointwise maximum of quadratic functions and thus PAMe is not applicable to ours. Second, we apply a quadratic extrapolation step to the update of the block variable $\bP$, which highly relies on the quadratic form of the objective function, while the authors  in \cite{wang2021linear} utilize the standard linear extrapolation step. Moreover, this extrapolation scheme also gives rise to the third difference between our convergence analysis and that in \cite{wang2021linear}. Specifically, the quadratic extrapolation step makes it more complicated to verify the decreasing property. 

\subsection{Notation and Definitions}

Let $\R^n$ be the $n$-dimensional Euclidean space. We write the matrices in bold capital letter like $\bQ$, vectors in bold low-case letters like $\bm{q}$, and scalars in plain letters like $q$. 
Given a matrix $\bX \in \R^{d\times n}$, we use $\|\bX\|_F$ to denote its Frobenius norm, $\|\bX\|$ its spectral norm, and $x_{ij}$ its $(i,j)$-th element. For any $x\in\R$, let
\[
\sign(x) \in \left\{
\begin{array}{c@{\,\,\,}l}
	\{x/|x|\}, & x \not= 0, \\
	\noalign{\smallskip}
	\{-1,1\}, & x = 0
\end{array}
\right.
\]
denote its sign function. 

Next, we introduce some standard concepts in non-smooth analysis for our development. The details can be found in, e.g., \cite{rockafellar2009variational}. For a non-empty closed set $\mS \subseteq \R^{p}$, the \emph{indicator function} $\delta_{\mS}: \R^{p} \rightarrow \{0,+\infty\}$ associated with $\mS$ is defined as 
\[
\delta_{\mS}(\bx) = \left\{
\begin{array}{c@{\,\,\,}l}
	0, & \bx \in \mS, \\
	+\infty, & \mbox{otherwise}.
\end{array}
\right.
\]

Let $f:\R^{p} \rightarrow (-\infty,+\infty]$ be a given function with  $\dom(f)=\{\bx \in \R^{p}: f(\bx)<+\infty\}$. The function $f$ is said to be \emph{proper} if $\dom(f)\not=\emptyset$. A vector $\bv\in\R^{p}$ is said to be a \emph{Fr\'{e}chet subgradient} of $f$ at $\bx \in \dom(f)$ if
\begin{equation} \label{eq:frech-subg}
	\liminf_{\by\rightarrow\bx, \atop \by\not=\bx} \frac{ f(\by) - f(\bx) - \langle \bv,\by-\bx \rangle }{ \|\by-\bx\|_F } \ge 0.
\end{equation}
The set of vectors $\bv \in \R^p$ satisfying~\cref{eq:frech-subg} is called the \emph{Fr\'{e}chet subdifferential} of $f$ at $\bx \in \dom(f)$ and denoted by $\widehat{\partial}f(\bx)$. The \emph{limiting subdifferential}, or simply the \emph{subdifferential}, of $f$ at $\bx\in \dom(f)$ is defined as  
\[
\partial f(\bx) = \left\{
\begin{split}
	&\ \bv\in\R^{p}: \exists \bx^k\rightarrow\bx, \, \bv^k \rightarrow \bv \,\\
	&\ \mbox{ with } \, f(\bx^k)\rightarrow f(\bx), \, \bv^k\in\widehat{\partial} f(\bx^k)
\end{split} \right\}.
\] 
Then, the limiting critical point of a function can be defined as follows. 
\begin{defi}[limiting critical point]\label{def:crit}
	Suppose that the function $f:\R^{p} \rightarrow (-\infty,+\infty]$ is proper and lower semicontinuous. We say that a point $\bx\in\R^{p}$ is  a  limiting critical point of $f$ if $\b0 \in \partial f(\bx)$.
\end{defi}
Remark that by the generalized Fermat rule (see, e.g.,~\cite[Theorem 10.1]{RW04}), a local minimizer of $f$ is a limiting critical point of $f$.
By convention, if $\bx\not\in\dom(f)$, then $\partial f(\bx) = \emptyset$. The \emph{domain} of $\partial f$ is defined as $\dom(\partial f) = \{\bx \in \R^p: \partial f(\bx) \not= \emptyset\}$. The limiting subdifferential of indicator function $\delta_{\mS}: \R^{p} \rightarrow \{0,+\infty\}$ is given as follows.
\[ 
\begin{aligned}
&\ \widehat{\partial}\delta_{\mS}(\bx) = \left\{ \bv\in\R^{p}: \limsup_{\by\rightarrow\bx, \, \by\in\mS \atop \by\not=\bx} \frac{\langle \bv, \by-\bx \rangle}{\|\by-\bx\|_F} \le 0 \right\}\\
&\ \quad\mbox{and}\quad \partial\delta_{\mS}(\bx) = \mN_{\mS}(\bx), \forall \bx \in \mS,
\end{aligned}
\]
where $\mN_{\mS}(\bx)$ is the \emph{normal cone} to $\mS$ at $\bx$.

\subsection{Organization}
The rest of this paper is organized as follows. In Section \ref{sec:algorithm}, we introduce the proposed PALMe method for solving Problem \eqref{PCA:L1}. Then, we prove the sufficient decrease property and relative error property in Section \ref{subsec:SD+RE}, estimate the K\L\ exponent of Problem \eqref{PCA:L1} in Section \ref{subsec:KL}, and prove the main theorem of this work in Section \ref{sec:convergence}. In Section \ref{sec:experiment}, we conduct numerical experiments to validate our theoretical results and compare the proposed approach to the existing approaches. We end with some concluding remarks in Section \ref{sec:conc}.

\section{Algorithm Design}\label{sec:algorithm}

We begin by reformulating Problem \eqref{PCA:L1} as a two-block form. Noting that $|x|=\max\{x,-x\}$ for any $x \in \R$, we can reformulate Problem \eqref{PCA:L1} as 
\begin{align}\label{PCA:L1-Re}
	\min\ & H(\bP,\bQ) := -\langle \bP, \bX^T \bQ \bQ^T \rangle \\
	\mathrm{s.t.}\ & \bP \in {\mB}(n,d),\ \bQ \in {\rm St}(d,K),
\end{align}
where $\langle \bA,\bB \rangle ={\rm tr}(\bA^T \bB)$ denotes the Euclidean inner product of two matrices $\bA,\bB$ of the same size and ${\mB}(n,d)=\left\{\bP \in \R^{n \times d}: p_{ij} \in \{\pm 1\},\ i=1,\ldots,n,\ j=1,\ldots,d\right\}$ denotes the set of all $n\times d$ matrices with elements $\pm 1$. Observing that this problem has two separate blocks of variables $\bP$ and $\bQ$, one can apply the PALM method (see, e.g., \cite{attouch2013convergence,bolte2014proximal,tseng2001convergence}) for solving it.  
Specifically, given the current iterate $\left(\bP^k,\bQ^k\right)\in \mB(n,d)\times {\rm St}(d,K)$, the method generates the next iterate $\left(\bP^{k+1},\bQ^{k+1}\right)\in {\mB}(n,d)\times {\rm St}(d,K)$ via
\begin{subequations}
	\begin{equation} 
		\bP^{k+1} \in \argmin  \left\{\begin{split}
			&\ -\langle \bP, \bX^T \bQ^k \bQ^{k^T}\rangle\\
			&\ +\frac{\alpha_k}{2} \|\bP-\bP^k\|^2_F:\ \bP \in \mB(n,d)
		\end{split} \right\}, \label{update:P}
	\end{equation}
	\begin{equation}
		\bQ^{k+1} \in \argmin \left\{\begin{split}
			&\ -\langle \bQ,(\bX\bP^{k+1}+\bP^{{k+1}^T} \bX^T)\bQ^k\rangle\\
			&\ +\frac{\beta_k}{2} \|\bQ-\bQ^k\|^2_F:\ \bQ \in {\rm St}(d,K)
		\end{split}\right\}\label{update:Q}, 
	\end{equation}
\end{subequations}
where $\alpha_k,\beta_k >0$ are step-size parameters. Recently, many different extrapolation techniques have been successfully applied to accelerate proximal algorithms for convex and non-convex optimization problems; see, e.g., \cite{nesterov1983method,beck2009fast,wen2018proximal,Li2015,lu2019enhanced,wang2021linear}. This motivates us to incorporate an extrapolation step into the update of $\bP$ to achieve possible acceleration of the PALM iterations. Specifically, we replace \eqref{update:P} with
\begin{equation}\label{update:P-extra}
	\begin{aligned}
		&\ \bE^k =\bQ^{k}\bQ^{k^T}+\gamma_k\left(\bQ^{k}\bQ^{k^T}-\bQ^{k-1}\bQ^{{k-1}^T}\right),\\
		&\ \bP^{k+1} \in \argmin \left\{\begin{split}
			&\ -\langle \bP, \bX^T \bE^k \rangle\\
			&\ +\frac{\alpha_k}{2} \|\bP-\bP^k\|^2_F:\ \bP \in {\mB}(n,d)
		\end{split}\right\},
	\end{aligned}
\end{equation}
where $\bE^k \in \R^{d\times K}$ is the point extrapolated from $\bQ^{k}\bQ^{k^T}$ and $\bQ^{k-1}\bQ^{{k-1}^T}$ and $\gamma_k \in [0,1)$ is the parameter controlling extrapolation step-size.   
It is worth noting that the above iterates \eqref{update:Q} and \eqref{update:P-extra} admit closed-form solutions. Specifically, the update \eqref{update:P} in its most simplified form reads as 
\begin{equation}\label{eq:l1_5}
	\bP^{k+1} \in \sign\left(\bP^k+\bX^T \bE^k /\alpha_k\right).
\end{equation}
On the other hand, the update \eqref{update:Q} is essentially an instance of the orthogonal Procrustes problem \cite{schonemann1966generalized}, whose solution is given by  
\[
\bQ^{k+1}=\bU^{k+1}\bV^{{k+1}^T},
\]
where $\bU^{k+1} \in \mathrm{St}(d,K)$ and $\bV^{k+1} \in \mathrm{St}(K,K)$ are obtained by a thin SVD $\bU^{k+1}\bSigma^{k+1}\bV^{{k+1}^T}=\bQ^k+(\bX\bP^{k+1}\bQ^k+{\bP^{k+1}}^T \bX^T \bQ^k)/\beta_k$. Now, we summarize the proposed method in Algorithm \ref{alg:palme}. One can verify that the per-iteration cost of the proposed method is $\mO(ndK+dK^2)$, which is cheap when $K \ll \min \left\{n,d\right\}$.  

\begin{algorithm}
	\caption{PALM with a quadratic extrapolation (PALMe) for L1-PCA }
	\label{alg:palme}	
	\SetAlgoLined
	\KwIn{$\bX \in \R^{d \times n},\ \bP^0 \in  {\mB}(n,d),\ \bQ^{-1}=\bQ^0 \in {\rm St}(d,K)$.}
	\For{$k=0,1,2,\cdots$}{
		choose step sizes $\alpha_k, \beta_k > 0$ and extraploation parameter $\gamma_k \in [0,1]$;\\
		set $\bE^k = \bQ^{k} \bQ^{k^T}+\gamma_k\left(\bQ^{k}\bQ^{k^T}-\bQ^{k-1}\bQ^{{k-1}^T}\right)$;\\
		pick $\bP^{k+1} \in \sign(\bP^k+\bX^T \bE^k /\alpha_k)$;\\
		compute a thin SVD $\bU^{k+1}\bSigma^{K+1}\bV^{{k+1}^T}=\bQ^k+\left(\bX\bP^{k+1}\bQ^k+\bP^{{k+1}^T} \bX^T \bQ^k\right)/\beta_k$;\\
		Set $\bQ^{k+1}=\bU^{k+1}\bV^{{k+1}^T}$;\\
		Terminate if stopping criteria are met;
	}
\end{algorithm}
 
We should emphasize that our algorithm has some essential differences with other block coordinate descent-type methods. First, most of the methods perform an extrapolation step on each block (see, e.g., \cite{le2020inertial,pock2016inertial,gao2020gauss,xu2017globally}), while PALMe only takes the extrapolation on one block $\bQ$. Experimental results in Section \ref{sec:experiment} validate the effectiveness of the latter approach for Problem \eqref{PCA:L1}. Second, although the update \eqref{update:P-extra} seems similar to that in \cite{wang2021linear}, our quadratic extrapolation step is totally new compared to the linear extrapolation step in \cite{wang2021linear}. As a result, the update of $\bQ$ is completely different. To the best of our knowledge, our work is the first one to design and analyze the quadratic extrapolation step for non-smooth and non-convex problems. 
 
\section{Convergence Analysis} 
 
Our goal in this section is to establish the convergence result of Algorithm \ref{alg:palme}. Generally, it is not easy to analyze the convergence behavior of algorithms for solving a non-smooth and non-convex problem. Fortunately, Attouch et al. \cite{attouch2013convergence} developed a unified framework that sheds light on analyzing the convergence behavior of proximal algorithms for solving non-smooth and non-convex problems. This framework has been widely used in the literature; see, e.g., \cite{wang2021linear,Liu2019QuadraticOW,lu2019enhanced,zeng2021analysis}. Specifically, one needs to verify that the sequence of iterates generated by the considered algorithm satisfies the \emph{sufficient decrease property} and the \emph{relative error property} in terms of a proper potential function. Moreover, one needs to verify that the considered potential function satisfies the \emph{K\L\ property with the associated exponent}. Note that the first two properties are algorithm-dependent, and we can directly verify them by analyzing the updates of Algorithm \ref{alg:palme}. By contrast, the last one depends on the function itself rather than the algorithm according to the definition of the K\L\ property. 
To simplify our development, we define
\begin{align}\label{PCA:l}
	\ell(\bQ) = -\|\bQ\bQ^T\bX\|_1 + \delta_{{\rm St}(d,K)}(\bQ),
\end{align}
and
\begin{align}\label{PCA:h}
	h(\bP,\bQ) = H(\bP,\bQ) + \delta_{\mB(n,K)}(\bP) + \delta_{{\rm St}(d,K)}(\bQ). 
\end{align}  
Moreover, we can define the critical points of $\ell$ as follows.
\begin{defi}[Critical point of $\ell$]\label{def:crit-l}
	We say that $\bQ \in {\rm St}(d,K)$ is a critical point of $\ell$ if 
	\begin{align}\label{crit:l}
		\z \in & -\bX\sign(\bX^T\bQ\bQ^T)\bQ - \sign(\bQ\bQ^T\bX)\bX^T\bQ \notag \\
		& + \mN_{{\rm St}(d,K)}(\bQ).
	\end{align}
\end{defi}
Invoking the subdifferential calculus rules in~\cite[Chapter 10B]{RW04}, we remark that this is a necessary condition  for local optimality. Besides, it should be noted that every limiting critical point of $\ell$ is a critical point of $\ell$, but the converse is not known to hold.

\subsection{Sufficient Decrease and Relative Error Properties}\label{subsec:SD+RE}

To prove the sufficient decrease and relative error properties, we need to choose a proper potential function. One immediate choice of the potential function is the function $h$ defined in \eqref{PCA:h}. However, due to the extrapolation step in Algorithm \ref{alg:palme}, it is not clear how to verify the mentioned properties in terms of $h$. Motivated by the potential functions constructed in \cite{wen2018proximal, pock2016inertial,wang2021linear}, we consider the potential function $\Phi_\beta(\bP,\bQ,\bQ'):\R^{n\times d} \times \R^{d\times K} \times \R^{d\times K} \rightarrow (-\infty,+\infty]$ defined by 
\begin{equation}
	\label{eq:potential}
	\Phi_\beta(\bP,\bQ,\bQ') = h(\bP,\bQ)+ \frac{\beta}{2}\|\bQ-\bQ'\|^2_F.
\end{equation}
Then, when the step-size and the extrapolation parameters in Algorithm \ref{alg:palme} are suitably chosen, we can verify the desired properties in terms of $\Phi_\beta$ as follows. 
\begin{lemma}\label{lem:suffi-decre}
	Let $\bC^k=\left(\bP^k,\bQ^k,\bQ^{k-1}\right)$ for all $k \geq 0$. Suppose that the step-size and extrapolation parameters in Algorithm \ref{alg:palme} satisfy for all $k \ge 0$,
	\begin{align}\label{pararms}
		&\ \alpha_* \le \alpha_k \le \alpha^*,\ \frac{3}{2}\beta_* + 2\|\bX\bP^k\|  \le \beta_k \le \beta^*,\\
		&\ 0 \le \gamma_k < \gamma^* = \min\left\{1,  \frac{\alpha_*\beta_*}{8\|\bX\|^2}\right\}
	\end{align}
	for some $\alpha_*,\alpha^*,\beta_*,\beta^* \in (0,\infty)$. Then, the following statements hold: \\
	(i) The sequence $\{\bC^k\}_{k\geq 0}$ is bounded.\\
	(ii) There exists a constant $\kappa_1 >0$ such that for all $k \ge 0$,
	\begin{align}\label{eq:suff-de}
		\Phi_{\beta_*}(\bC^{k+1})-\Phi_{\beta_*}(\bC^{k}) \le -\kappa_1 \|{\bC}^{k+1}-{\bC}^k\|_F^2.
	\end{align}
	(iii) There exists a constant $\kappa_2 >0$ such that for all $k \ge 0$,
	\begin{align}\label{eq:rela-err}
		\dist\left(\b0,\partial \Phi_{\beta_*}(\bC^{k+1}) \right) \le \kappa_2 \|{\bC}^{k+1}-{\bC}^k\|_F. 
	\end{align}
\end{lemma}
We defer the detailed proof to Section \ref{first-appendix} of the appendix. Thanks to the boundness of ${\mB}(n,d)$ and ${\rm St}(d,K)$, one can easily verify the boundness of the sequence $\{\bC^k\}_{k\geq 0}$ as in (i). We refer to (ii) as \emph{sufficient decrease property}, which follows from the updates in \eqref{update:Q} and \eqref{update:P-extra}, and the Lipschitz smoothness of $H(\bP,\bQ)$ on a bounded set. It ensures that the potential function $\Phi_{\beta_*}(\bC^{k})$ is monotonically decreasing. This, together with (i), implies that the sequence $\{\Phi_{\beta_*}(\bC^{k})\}_{k\ge 0}$ converges and $\|{\bC}^{k+1}-{\bC}^k\|_F$ goes to zero as $k$ increases. We refer to (iii) as \emph{relative error property}, which is also known as {\it safeguard} in \cite{zhou2017unified,LYS17}. It is proved by the optimality condition of the updates in \eqref{update:Q} and \eqref{update:P-extra}. Armed with this lemma, we can conclude that every accumulation point of $\{{\bC}^k\}_{k\ge 0}$ is a critical point of $\Phi_{\beta_*}$, which is essentially \emph{subsequence convergence}.  Remark that compared to the proof in \cite[Proposition 6]{wang2021linear}, this proof is different and more involved. Indeed, the objective function of Problem \eqref{PCA:L1-Re} is quadratic in terms of $\bQ$, while that is linear in \cite{wang2021linear}. Moreover, our extrapolation step in \eqref{update:P-extra} is non-linear, while that is linear in \cite{wang2021linear}. 

\subsection{K\L\ Property} \label{subsec:KL}

In this subsection, we verify the K\L\ property of the constructed potential function $\Phi_{\beta}$ and derive the corresponding K\L\ exponent. For the definition of the K\L\ property, we refer the reader to \cite[Definition 2.5]{attouch2013convergence} for the details. The K\L\ property  is widely used for studying convergence behavior of various first-order methods; see, e.g., \cite{attouch2010proximal, attouch2013convergence, bolte2014proximal}. It is known that the K\L\ property holds for many extended-valued lower-semicontinuous functions, such as semi-algebraic functions. According to \cite[Section 4.3]{attouch2010proximal} and \cite[Theorem 3]{bolte2014proximal}, one can easily verify that $\Phi_{\beta}$ satisfies the K\L\ property.  Here, we are more interested in the K\L\ exponent of $\Phi_{\beta}$, since it determines the convergence rate of the proposed method.   
\begin{defi}[K\L\ exponent]\label{def:KL-exponent}
 Suppose that $f:\R^d \rightarrow (-\infty,\infty]$ is proper and lower semicontinuous. We say that the function $f$ has a \emph{K\L\ exponent} of $\theta \in [0,1)$ at the point $\bar{\bx} \in \dom(\partial f)$ if there exist constants $\epsilon, \eta>0$, $\nu \in (0,+\infty]$ such that
	\[ \dist(\b0,\partial f(\bx)) \ge \eta( f(\bx) - f(\bar{\bx}) )^{\theta} \]
	whenever $\| \bx - \bar{\bx} \| \le \epsilon$ and $f(\bar{\bx}) < f(\bx) < f(\bar{\bx}) + \nu$.
\end{defi}	
Now, we are devoted to estimating the K\L\ exponent of $\Phi_{\beta}$. Let $\bP_1,\ldots,\bP_{2^{nd}}$ be an enumeration of the elements in $\mB(n,d)$. By definition of the $\ell_1$-norm, we can express the objective function $\ell$ of Problem \eqref{PCA:l} as the pointwise minimum of finitely many proper and lower semicontinuous functions:
\[
\ell(\bQ) = \min_{i \in \{1,\ldots,2^{nd}\}} \big\{ \underbrace{\langle \bX\bP_i, \bQ\bQ^\top \rangle + \delta_{{\rm St}(d,K)}(\bQ)}_{\ell_i(\bQ)} \big\}.
\]    
Using the same argument in \cite[Section 3.1]{wang2021linear}, we can elucidate that estimating the K\L\ exponent of $\ell$ boils down determining the K\L\ exponent of $\ell_1,\dots,\ell_{2^{nd}}$. 
Based on \cite[Theorem 1]{Liu2019QuadraticOW} and \cite[Lemma 2.1]{li2018calculus},  we can characterize the K\L\ exponent of the quadratic optimization problem with orthogonality constraint (QP-OC).
\begin{fact}\label{fact:QPOC}
	For the QP-OC problem 
	\begin{align*}
		\min\{\langle\bA, \bQ\bB\bQ^T \rangle:\ \bQ\in {\rm St}(d,K)\},
	\end{align*}
	where $\bA \in \R^{d\times d}$ and $\bB \in \R^{K \times K}$ are symmetric, the K\L\ exponent  is $1/2$. 
\end{fact}
Using \cite[Theorem 3.6]{li2018calculus}, the K\L\ exponent of $h$, and relations between $h$ and $\Phi_{\beta_*}$, we can obtain the following corollary. The detailed proof is provided in Section \ref{appendix:coro-KL} of the appendix. 
\begin{coro}\label{coro:KL}
	The K\L\ exponent of $\Phi_{\beta_*}$ is $1/2$. 
\end{coro}


\subsection{Linear Convergence of Algorithm \ref{alg:palme} and Properties of Limit Points} \label{sec:convergence}
Equipped with Lemma \ref{lem:suffi-decre}, Corollary \ref{coro:KL}, and \cite[Theorem 2.9]{attouch2013convergence}, we are ready to prove that the sequence $\{\bC^k\}_{k\ge 0}$ converges at least linearly  to a limiting critical point of $\Phi_{\beta_*}$. However, our goals are to prove the convergence rate of Algorithm \ref{alg:palme} and figure out under what condition the limit point of the sequence $\{\bQ^k\}$ is a critical point of Problem \eqref{PCA:L1}. To tackle the issues, we need to characterize the relationships among $\ell$, $h$, and $\Phi_{\beta}$. 
\begin{lemma}\label{relat:h-l}
	Let $\beta > 0$ be given and $\ell$, $h$, and $\Phi_{\beta}$ be respectively defined in \eqref{PCA:l}, \eqref{PCA:h}, and \eqref{eq:potential}. Then, the following statements hold: \\
	(i) If $(\bP,\bQ)$ with $\bP \in \sign(\bX^T\bQ\bQ^T)$ is a limiting critical point of $h$, then $\bQ \in {\rm St}(d,K)$ is a critical point of $\ell$.  \\
	(ii) Suppose that $(\bP,\bQ,\bQ^\prime) \in \mB(n,d) \times {\rm St}(d,K) \times {\rm St}(d,K)$ is a limiting critical point of $\Phi_\beta$. Then, we have $\bQ = \bQ^\prime$. Moreover, $(\bP,\bQ,\bQ)$ is a limiting critical point of $\Phi_\beta$ if and only if $(\bP,\bQ)$ is a limiting critical point of $h$.  
\end{lemma}

We defer the detailed proof to Section \ref{second-appendix} of the appendix. Armed with this lemma, \ref{lem:suffi-decre}, Corollary \ref{coro:KL}, and \cite[Theorem 2.9]{attouch2013convergence}, we can establish the main theorem  of this work, which provides convergence analysis of the proposed method.
\begin{thm}\label{thm:2}
	Let $\{(\bP^k,\bQ^k)\}_{k\ge0}$ be the sequence of iterates generated by Algorithm \eqref{alg:palme} under the conditions (i) $\alpha_* \leq \alpha_k \leq \alpha^*$ for some $\alpha^*, \alpha_* \in (0,+\infty)$, (ii) $3\beta_*/2 + 2\|\bX\bP^k\|  \le \beta_k \le \beta^*$ for some $\beta^*,\beta_* \in (0,+\infty)$, and (iii) $\gamma_k$ satisfy $0 \le \gamma_k < \gamma^* = \min\left\{1,  \alpha_*\beta_*/({8\|\bX\|^2})\right\}$. Then, the sequence $\{(\bP^k,\bQ^k)\}_{k\ge0}$ converges at least linearly to a limiting critical point $(\bP^*,\bQ^*)$ of Problem \eqref{PCA:L1-Re}. Moreover, if $\alpha_k=\alpha_*$ for  all $k \geq 0$ such that
	\begin{equation}\label{alpha:low}
		0 <  \alpha_*  < \min\left\{\begin{split}
			&\  |(\bX^T\bQ^*\bQ^{*\top})_{ij}| :(\bX^T\bQ^*\bQ^{*T})_{ij} \not= 0, \\
			&\  \quad i=1,\dots,n, \, j=1,\dots,d
		\end{split}\right\},
	\end{equation}
	then $ \bQ^*$ is a critical point of Problem \eqref{PCA:L1}. 
\end{thm}
\begin{proof}
	Let $\bC^k=\left(\bP^k,\bQ^k,\bQ^{k-1}\right)$ for all $k \geq 0$. It follows from Lemma \ref{lem:suffi-decre}, Corollary \ref{coro:KL}, and \cite[Theorem 1]{bolte2014proximal} that the sequence $\{\bC^k\}_{k \ge 0}$ converges at least linearly to a limiting critical point $\left(\bP^*,\bQ^*,\bQ^*\right)$ of $\Phi_{\beta_*}$. This, together with (ii) of Lemma \ref{relat:h-l}, gives that the sequence $\{(\bP^k,\bQ^k)\}_{k\ge 0}$ converges at least linearly to the limiting critical point $(\bP^*,\bQ^*)$ of $h$.  
	
	Suppose that $\alpha_k = \alpha_*$ for all $k \ge 0$. According to \eqref{update:P-extra} and \eqref{eq:l1_5}, we have $\bE^k = \bQ^{k}\bQ^{k^T}+\gamma_k (\bQ^{k}\bQ^{k^T}-\bQ^{k-1}\bQ^{{k-1}^T})$ and $\bP^{k+1} \in \sign(\bP^k + \bX^T\bE^k/\alpha_k)$.  This, together with $(\bP^k,\bQ^k) \rightarrow (\bP^*,\bQ^*)$ and $\alpha_k = \alpha_*$, yields that  
	\begin{align*}
		\bP^* \in \sign(\bP^*+\bX^T\bQ^*\bQ^{*^T}/\alpha_*). 
	\end{align*}
	Combining this with the result that $(\bP^*,\bQ^*)$ is a limiting critical point of $h$ gives  
	\begin{equation*}
		\begin{split}
			\z \in &\ -\bX\sign(\bP^*+\bX^T\bQ^*\bQ^{*^T}/\alpha_*)\bQ^*\\
			&\ - \sign(\bP^{*^T}+\bQ^*\bQ^{*^T}\bX^T/\alpha_*)\bX\bQ^* + \mN_{{\rm St}(d,K)}(\bQ^*).
		\end{split} 
	\end{equation*}
	In particular, noting that $\bP^* \in \mB(n,d)$ and $\alpha_*$ satisfies \eqref{alpha:low}, we have $\sign(\bP^*+\bX^T\bQ^*\bQ^{*^T}/\alpha_*) \subseteq \sign( \bX^T\bQ^*\bQ^{*^T} )$. These further implies 
	\begin{equation*}
		\begin{split}
			\z \in &\ -\bX\sign( \bX^T\bQ^*\bQ^{*^T} )\bQ^* - \sign( \bQ^*\bQ^{*^T}\bX^T )\bX^T\bQ^* \\
			&\ + \mN_{{\rm St}(d,K)}(\bQ^*).
		\end{split}
	\end{equation*}
	Then, we complete the proof. 
\end{proof}

The first part of this theorem  demonstrates that under some conditions imposed on step-size and extrapolation parameters, the iterates generated by Algorithm \ref{alg:palme} linearly converges to a limiting critical point of the reformulated L1-PCA problem \eqref{PCA:L1-Re}. The conditions imposed on $\alpha_k,\beta_k$ and $\gamma_k$ are introduced by proving the sufficient decreasing of objective function value. Interestingly, we found that in experiments, these parameters can be chosen flexibly even if they may violate aforementioned conditions. We leave this theory-practice gap for further research. 
The second part of Theorem \ref{thm:2} shows that the condition \eqref{alpha:low} is sufficient to guarantee that the point returned by Algorithm \ref{PCA:L1} is not only a limiting critical point of Problem \eqref{PCA:L1-Re}, but also a critical point of Problem \eqref{PCA:L1}. It is worth noting that we can efficiently verify whether this condition holds after obtaining a limit point $\bQ^*$.

\section{Experiment results}\label{sec:experiment}

In this section, we conduct numerical experiments on synthetic and real datasets to demonstrate the efficacy of our studied approach. We first compare the proposed method with other existing ones for solving Problem \eqref{PCA:L1} in terms of convergence performance and solution quality in Section \ref{subsec:conv-perf} and clustering accuracy in Section \ref{subsec:clus-accu}. In Section \ref{subsec:image-re}, we then compare the image reconstruction performance of our studied approach with the surveyed other L1-PCA approaches in Section \ref{sec:intr}.
All the experiments are performed on a PC running Windows 10 with an Intel\textsuperscript{\textregistered} Core\texttrademark\, i5-8600 3.10GHz CPU and 16GB memory. Our codes are implemented in MATLAB R2021a and  can be found at \url{https://github.com/TaoliZheng/RIL1PCA}. 

\begin{table*}[t]
	\caption{Step-size parameters of the tested methods}
	\label{table-1}\vspace{-0.15in}
	\begin{center}
		\begin{tabular}{c | c  c  c  c c c c r}
			\hline
			\footnotesize $(\alpha,\beta )$ & \footnotesize PALMe & \footnotesize PALM & \footnotesize iPALM & \footnotesize GiPALM  & \footnotesize pDCAe \\
			\hline
			\footnotesize synthetic $(n,d) = (5000,1000)$
			& \footnotesize $(10^{-7}, 100)$ & \footnotesize $(10^{-7}, 10)$ & \footnotesize $(10^{-7}, 10)$ & \footnotesize $(10^{-7}, 100)$ & \footnotesize $(-, 0.05)$  \\ 
			\footnotesize synthetic  $(n,d) = (1000,5000)$ & \footnotesize $(10^{-6}, 1)$ & \footnotesize $(10^{-6}, 10)$ & \footnotesize $(10^{-6}, 0.1)$ & \footnotesize $(10^{-6}, 1)$ & \footnotesize $(-, 0.1)$   \\
			\footnotesize \emph{colon-cancer} $(n,d) = (62,2000)$ & \footnotesize $(10^{-10}, 100)$ & \footnotesize $(10^{-10}, 500)$ & \footnotesize $(10^{-10}, 500)$ & \footnotesize $(10^{-10}, 100)$ & \footnotesize $(-, 100)$   \\ 
			\hline
		\end{tabular}
	\end{center}
\end{table*}

\begin{table*}[!htbp]
	\caption{TEV of the tested methods}
	\label{table-2}\vspace{-0.15in}
	\begin{center}
		\begin{tabular}{c | c  c  c  c c c c r}
			\hline
			\footnotesize $(\alpha,\beta )$ & \footnotesize PALMe & \footnotesize PALM & \footnotesize iPALM & \footnotesize GiPALM  & \footnotesize pDCAe \\
			\hline
		  \footnotesize synthetic  $(n,d) = (5000,1000)$ & \footnotesize \textbf{0.978176} & \footnotesize 0.973894 & \footnotesize 0.973955 & \footnotesize 0.974290 & \footnotesize 0.973938 \\ 
			 \footnotesize  synthetic  $(n,d) = (1000,5000)$ & \footnotesize \textbf{0.955969} & \footnotesize  0.940970 & \footnotesize 0.940904 & \footnotesize 0.943625 & \footnotesize 0.941027   \\
			\emph{colon-cancer}  $(n,d) = (62,2000)$ & \footnotesize 0.925389 & \footnotesize 0.921079 & \footnotesize 0.921001 & \footnotesize \textbf{0.928077} & \footnotesize 0.924373  \\ 
			\hline
		\end{tabular}
	\end{center}
\end{table*}

\begin{figure*}[!htbp]
	\begin{subfigure}[b]{0.3\textwidth}
		\centering
		\includegraphics[width=\textwidth]{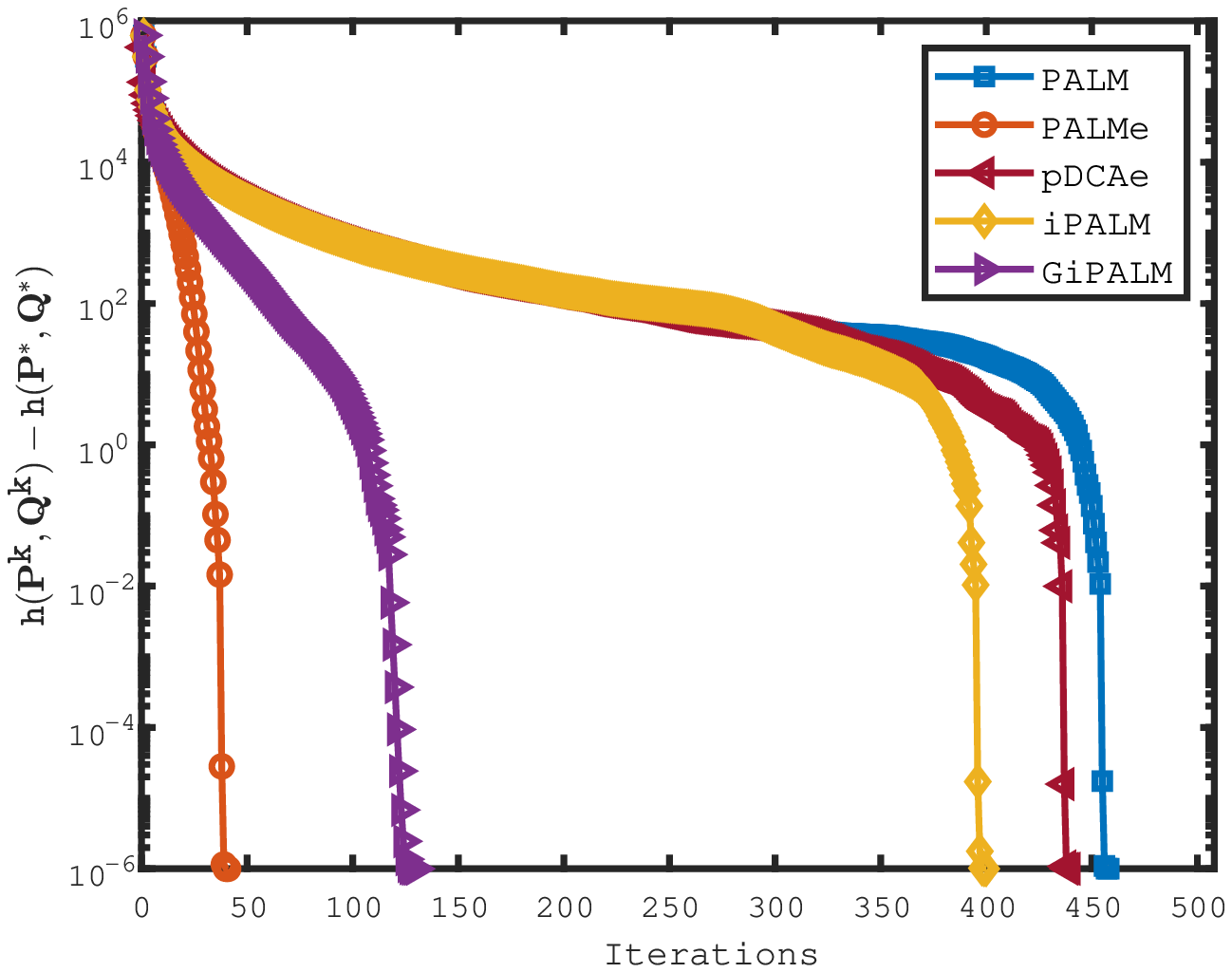}
		\caption{synthetic:  $(n,d)=(1000,5000)$}
	\end{subfigure}
	\begin{subfigure}[b]{0.3\textwidth}
		\centering
		\includegraphics[width=\textwidth]{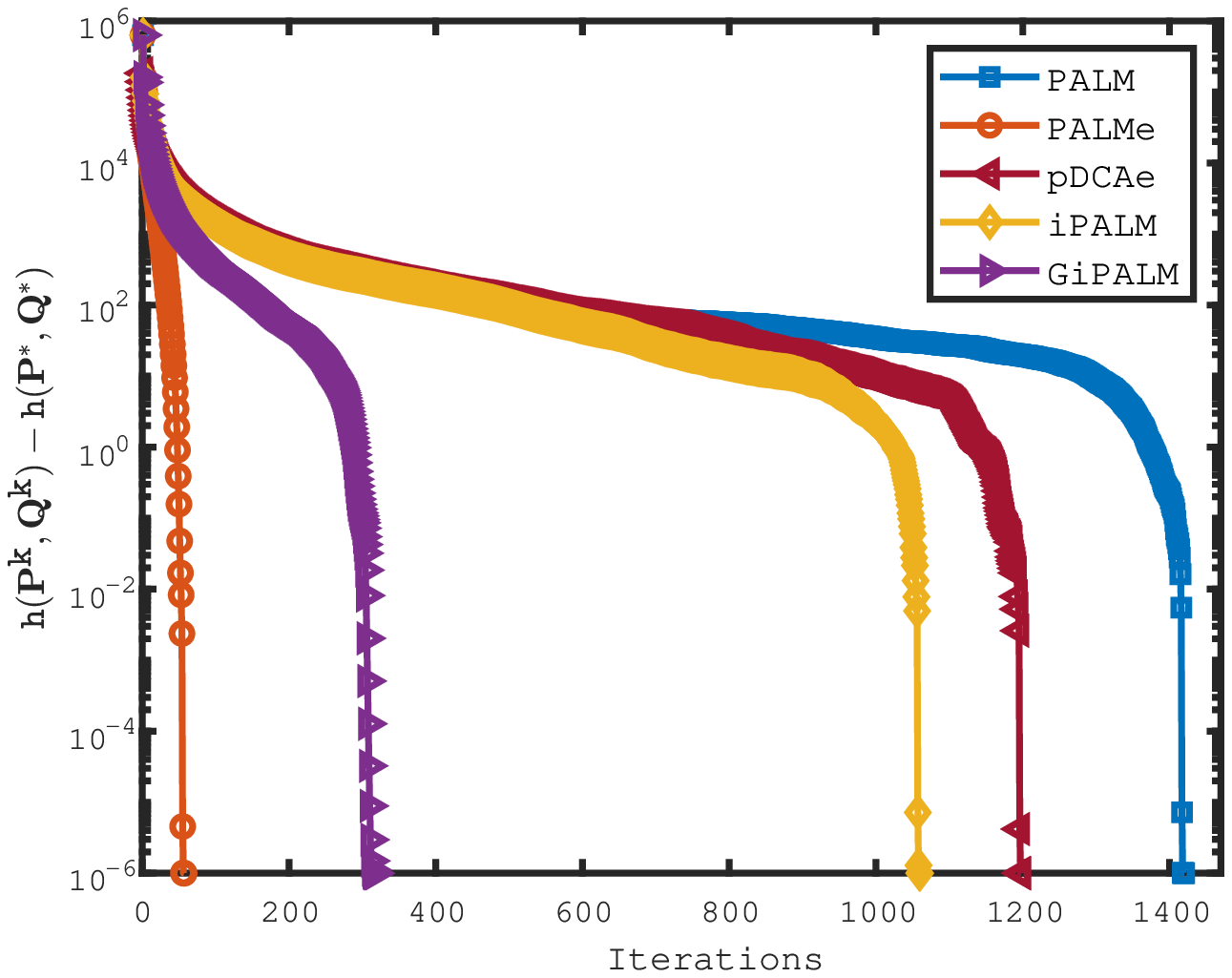}
		\caption{synthetic:  $(n,d)=(5000,1000)$}
	\end{subfigure}
    \begin{subfigure}[b]{0.3\textwidth}
    	\centering
    	\includegraphics[width=\textwidth]{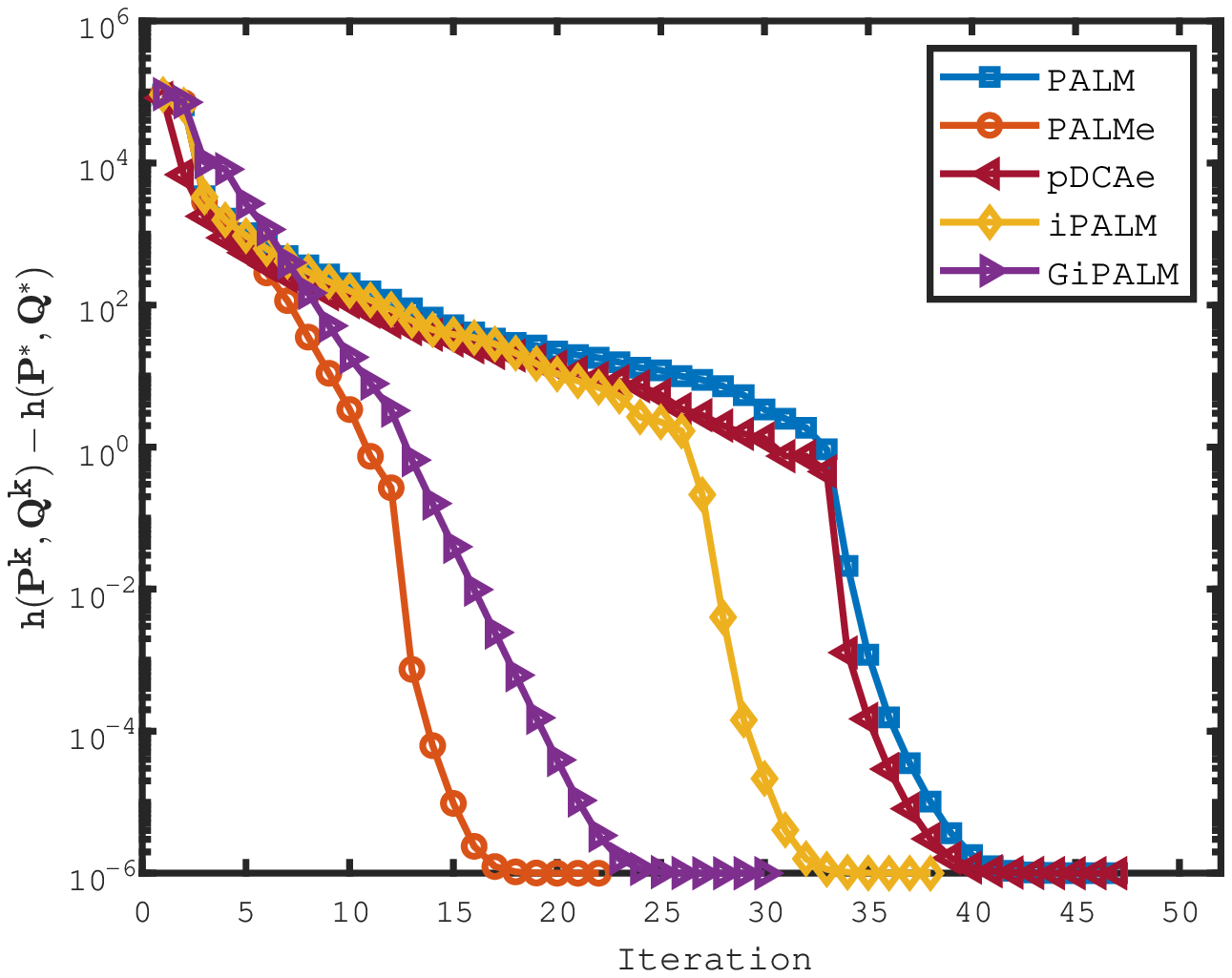}
    	\caption{real \emph{colon-cancer}: $(n,d) = (62,2000)$}
    \end{subfigure}
    \caption{Convergence performance of function values: The $x$-axis is number of iterations, and the $y$-axis is function value gap $h(\bP^k,\bQ^k)-h(\bP^*,\bQ^*)$, where $(\bP^*,\bQ^*)$ is the last iterate of the tested method. }
	\label{fig:1}
\end{figure*}

\subsection{Convergence Performance and Solution Quality}\label{subsec:conv-perf}

We first examine the convergence performance and solution quality of the proposed method on synthetic and real  datasets. We also compare it with some existing methods that can be applied to solve Problem \eqref{PCA:L1} or its reformulation \eqref{PCA:L1-Re}, which includes the standard PALM in \cite{bolte2014proximal}, the inertial proximal alternating linearized minimization (iPALM) method in \cite{pock2016inertial}, the Gauss-Seidel-type iPALM (GiPALM) method in \cite{gao2020gauss}, and the proximal difference-of-convex with extrapolation (pDCAe) method in \cite{wen2018proximal}. As demonstrated in \cite{kim2019simple,wang2021linear}, we can use the following measure named \emph{total explained variation} (TEV) to 
compare the quality of solutions returned by the tested methods: 
\begin{align*}
	{\rm TEV} = \frac{\|\bX^T\bQ\|_F^2}{\|\bX^T\bar{\bQ}\|_F^2},
\end{align*} 
where $\bQ \in \R^{d\times K}$ is the solution returned by the tested method and $\bar{\bQ} \in \R^{d\times K}$ is the matrix formed by the eigenvectors associated with the leading $K$ eigenvalues of  $\bX\bX^T$. In general, it
is assumed that the larger TEV leads to the better solution.

In the tests, we employ the fixed effect model in \cite{baccini1996l1} to generate the synthetic data. We use the same approach as that in \cite{wang2021linear} to generate a $K$-dimensional subspace and a Laplacian noise with mean  zero and variance $\sigma^2$.  As a result, the generated data points lie near the  subspace perturbed by the noise.   We set $\sigma=0.5, K=50$ and generate two synthetic datasets with dimensions $\left(n,d\right)=\left(1000,5000\right)$ and $\left(n,d\right)=\left(5000,1000\right)$, respectively. As for the real dataset, we use the dataset \emph{colon-cancer} downloaded from LIBSVM \cite{chang2011libsvm}\footnote{\url{https://www.csie.ntu.edu.tw/~cjlin/libsvmtools/datasets/}\label{web}} with dimensions $\left(n,d\right)=\left(62,2000\right)$  and set $K=20$. 

\begin{figure*}[!htbp]
	\begin{subfigure}[b]{0.3\textwidth}
		\centering
		\includegraphics[width=\textwidth]{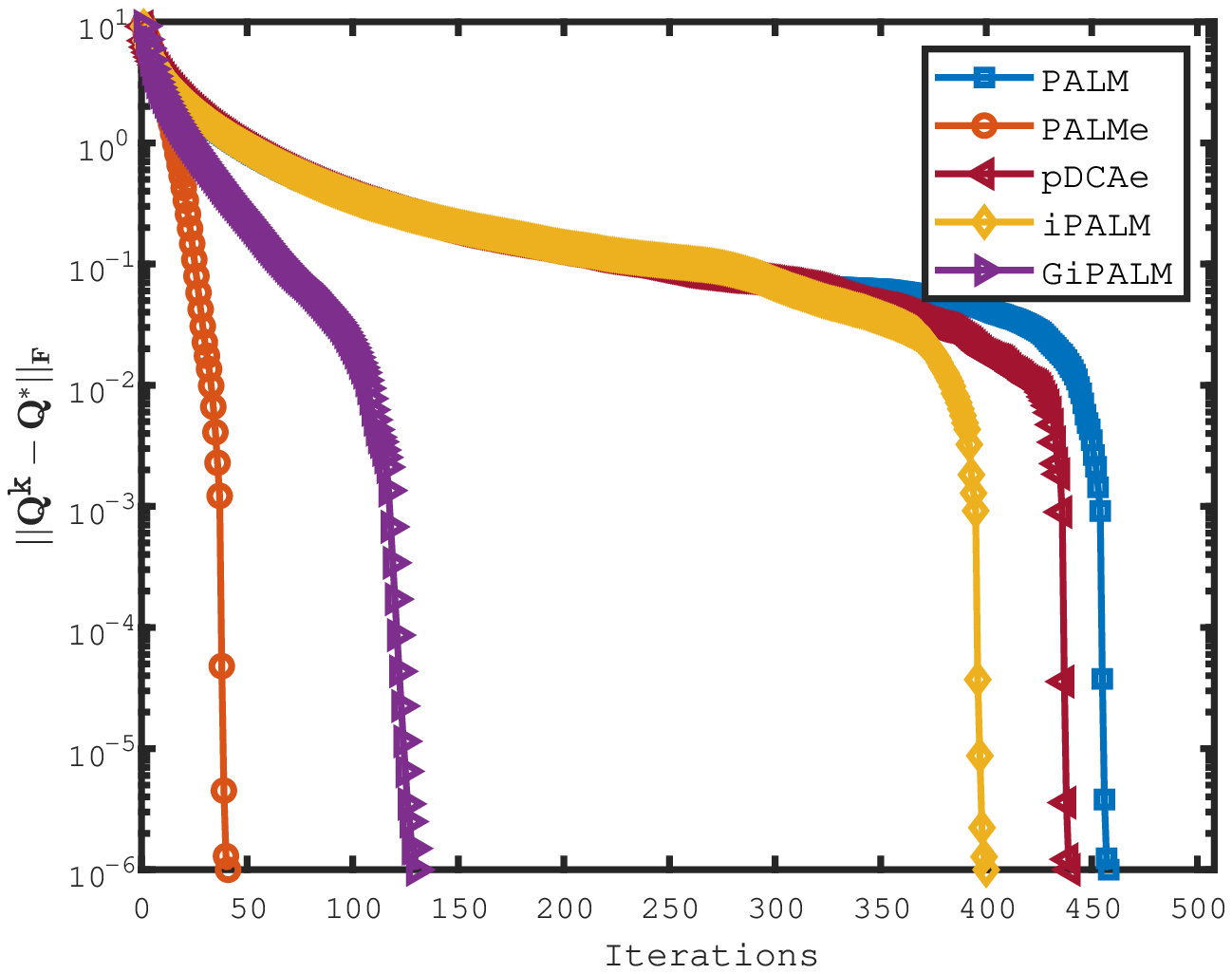}
		\caption{synthetic: $(n,d)=(1000,5000)$}
	\end{subfigure}
	\begin{subfigure}[b]{0.3\textwidth}
		\centering
		\includegraphics[width=\textwidth]{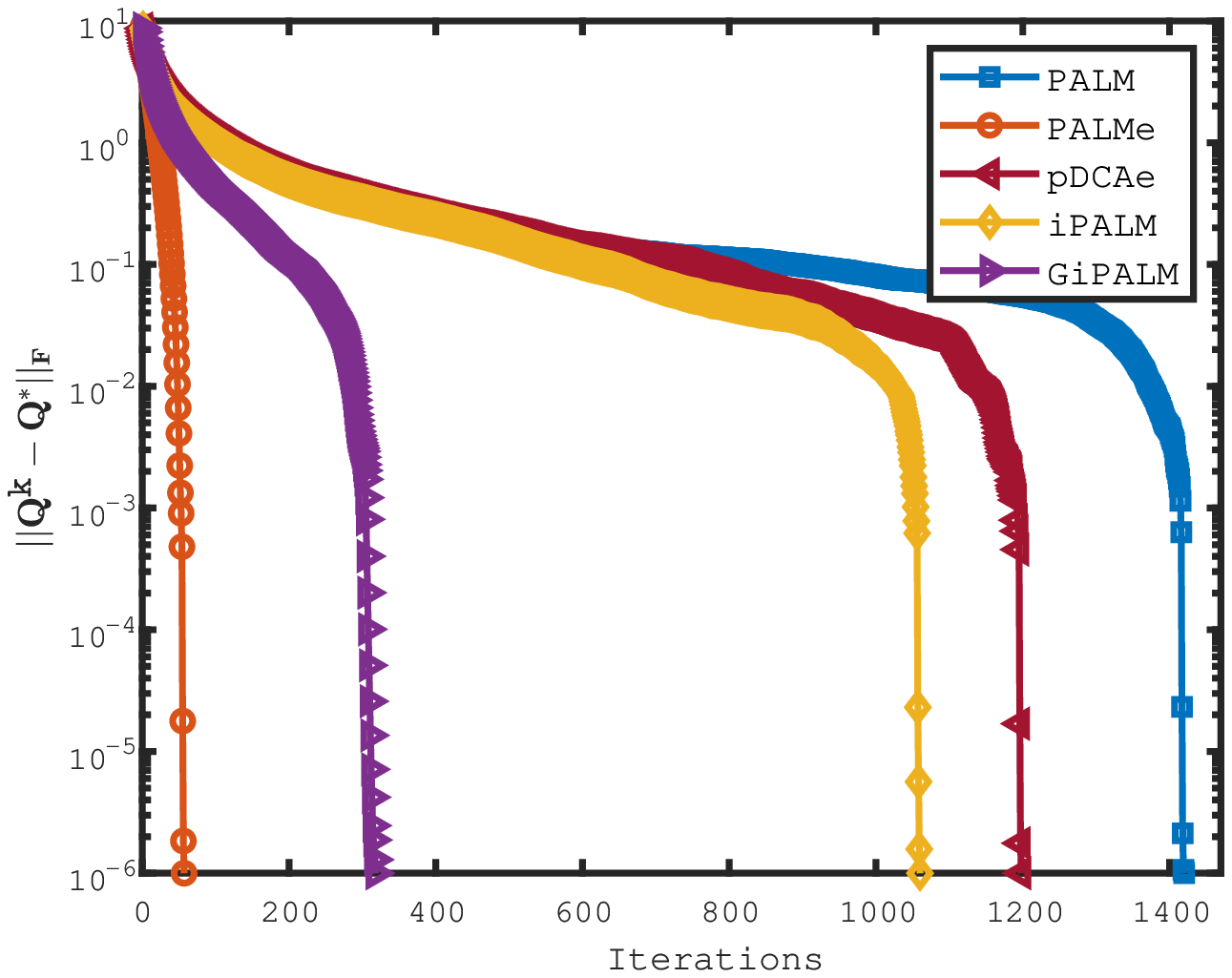}
		\caption{synthetic: $(n,d)=(5000,1000)$}
	\end{subfigure}
    \begin{subfigure}[b]{0.3\textwidth}
    	\centering
    	\includegraphics[width=\textwidth]{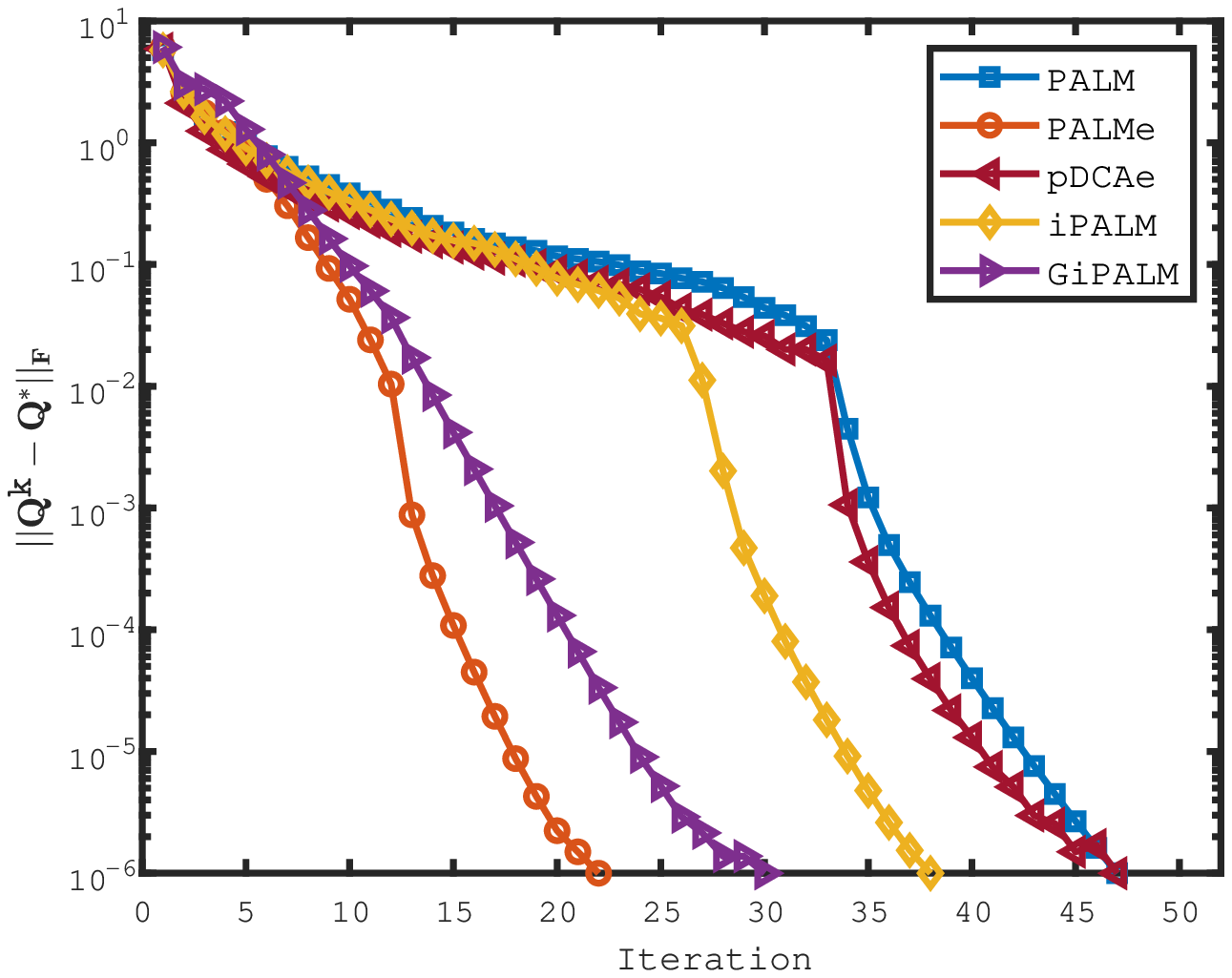}
    	\caption{\emph{colon-cancer}: $(n,d)= (62,2000)$}
    \end{subfigure}
	\caption{Convergence performance of iterates: The $x$-axis is number of iterations, and the $y$-axis is iterate gap $\|\bQ^k-\bQ^*\|$, where $\bQ^*$ is the last iterate of the tested method.}
	\label{fig:2}
\end{figure*}

We then present the setting of parameters for different algorithms in the tests. To be fair, all the parameters are tuned based on numerous runs to balance the solution quality and running time. We first specify the step-size parameters of the tested methods. We set $(\alpha_k,\beta_k)=(\alpha,\beta)$  for all $k \ge 0$, whose values are provided in Table \ref{table-1}. Note that it is not required to choose $\alpha_k$ for pDCAe since it only has one block of variables. We next specify the extrapolation parameters of the tested methods. For PALMe, we set it as $1$. We observe that it works surprisingly well, even though such a choice could violate the condition in Theorem \ref{thm:2}. For pDCAe, we set the extrapolation parameter as 0.2. For iPALM, we set the extrapolation parameters for updating the block variables $\bP$ and $\bQ$ both as 0.2. For GiPALM, these two parameters are respectively set as $1/2$ and $1/4$  for all $k \ge 0$. In each test, we employ the same starting point for the tested algorithms. For each algorithm, we terminate it when the Frobenious norm of the difference between two consecutive iterates is less than $10^{-6}$. 


To compare the convergence performance of the tested algorithms, we plot the the function value gap $h(\bP^k,\bQ^k)-h(\bP^*,\bQ^*)$ and iterates gap $\|\bQ^k - \bQ^*\|_F$ against the iteration number for all the tested algorithms in Figs. \ref{fig:1} and \ref{fig:2}, respectively. For each tested algorithm, we choose its last iterate as $(\bP^*,\bQ^*)$. According to these figures, we can observe that the sequences of function value gaps and iterate gaps converge linearly, which supports our theoretical result in Theorem \ref{thm:2}. Moreover, it can be observed that PALMe enjoys a substantially faster linear convergence compared to the other methods.
To compare the solution quality, we report the value of TEV of the tested methods in Table \ref{table-2} averaged over 10 runs. We can observe that the solution quality of PALMe is comparable to those of the other methods. 

\subsection{Clustering Accuracy}\label{subsec:clus-accu}

In this subsection, we compare the clustering performance of our proposed method with PALM, iPALM, GiPALM, and pDCAe, which are introduced in Section \ref{subsec:conv-perf}. A good way of evaluating the clustering performance of a algorithm for solving Problem \eqref{PCA:L1} is to study its clustering accuracy when applied to clustering on a subspace; see, e.g., \cite{ding2006r, wang2021linear}. Specifically, we obtain a label vector by applying the \emph{k-means} clustering to the projections of data points onto the subspace returned by a algorithm for solving Problem \eqref{PCA:L1} and compute the clustering accuracy by comparing the returned label and the true label. In our experiments, we use the real-world datasets \emph{a6a}, {\em a9a}, {\em gisette}, and {\em ijcnn1} downloaded from LIBSVM \cite{chang2011libsvm}, whose dimensions can be found in Table \ref{table-3}. The dimension  $K$ of the subspace used by L1-PCA  is chosen such that $\sum_{k=1}^{K} \sigma^2_k \geq 0.8 \sum_{k=1}^{p} \sigma^2_k$, where $p=\min\{n,d\}$ and $\sigma_1\ge \dots \ge \sigma_p \ge 0$ are singular values of $\bm{X}$. We list the value of $K$ for each dataset in Table \ref{table-3}. When the difference between two consecutive iterates is less than $10^{-6}$, we stop the iterations.

In the tests, we set the step-size and extrapolation parameters of the tested algorithms as in Table \ref{table-3}. The fixed step-size is still used in experiments. 
To better evaluate the clustering performance, each algorithm is tested for 10 times. Then, we plot the clustering accuracy and the running time for the tested algorithms in Fig. \ref{fig:3} averaged over 10 runs. It is worth noting that the point in the top-left corner of the axes converges faster and achieves higher clustering accuracy than that in the bottom-right corner. 
Then, we can observe from Fig. \ref{fig:3} that PALMe can generally  achieve a comparable clustering accuracy to those of the other methods with less running time, which demonstrates its efficacy.

\begin{figure*}[!htbp]
	\begin{subfigure}[b]{0.22\textwidth}
		\centering
		\includegraphics[width=\textwidth]{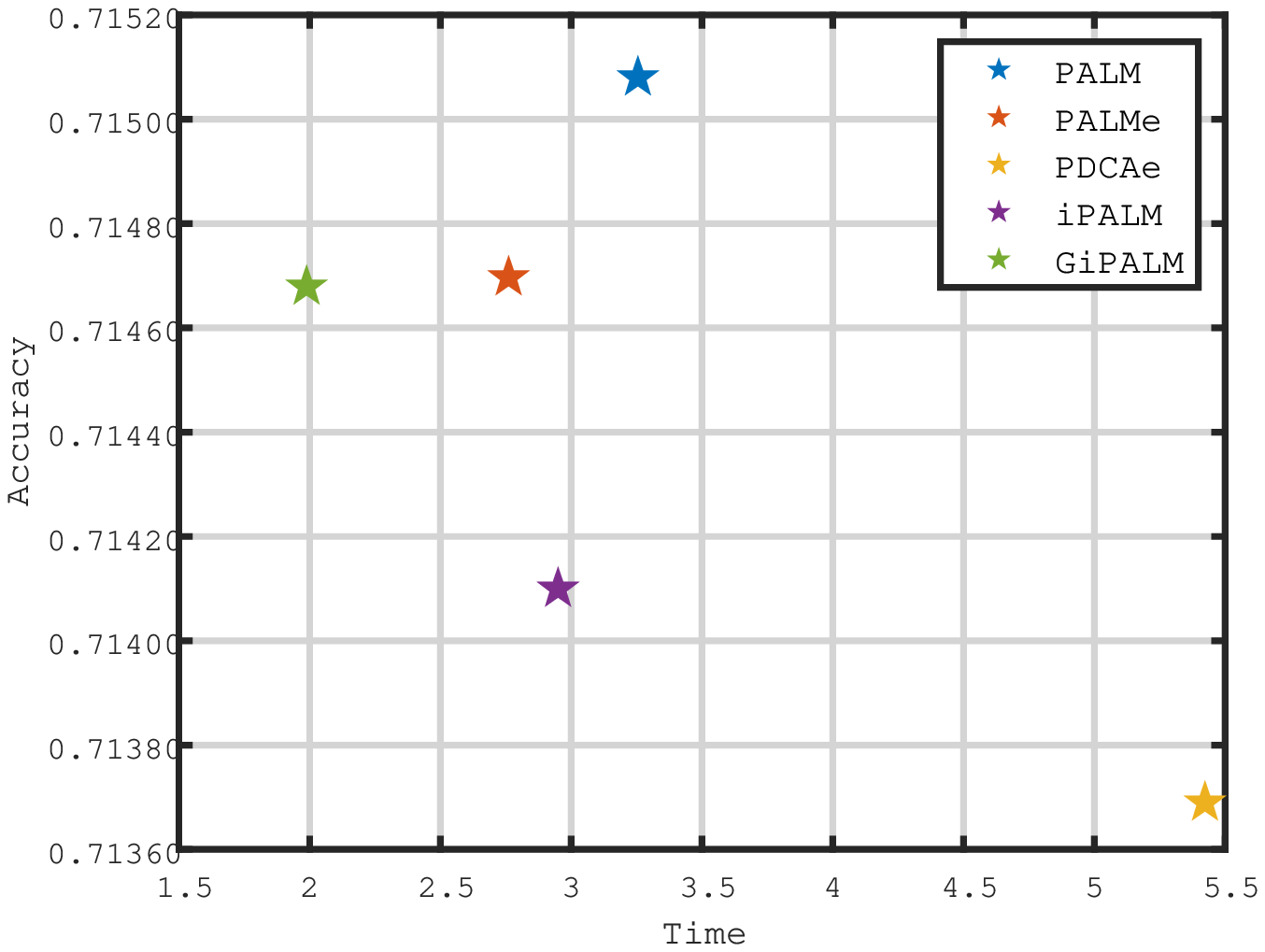}
		\caption{  $\emph{a6a}$}
	\end{subfigure}
	\begin{subfigure}[b]{0.22\textwidth}
		\centering
		\includegraphics[width=\textwidth]{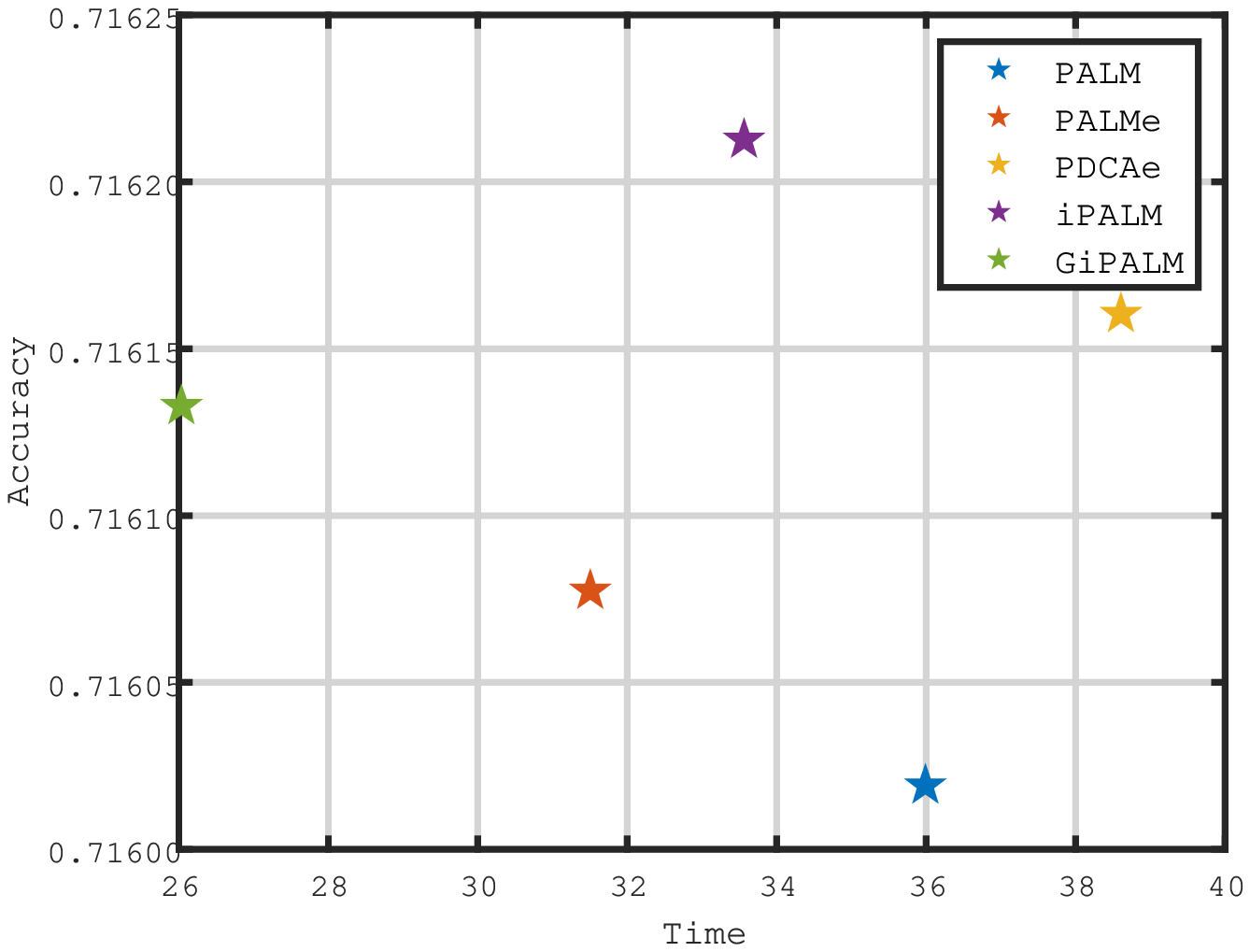}
		\caption{ $\emph{a9a}$}
	\end{subfigure}
	\begin{subfigure}[b]{0.22\textwidth}
		\centering
		\includegraphics[width=\textwidth]{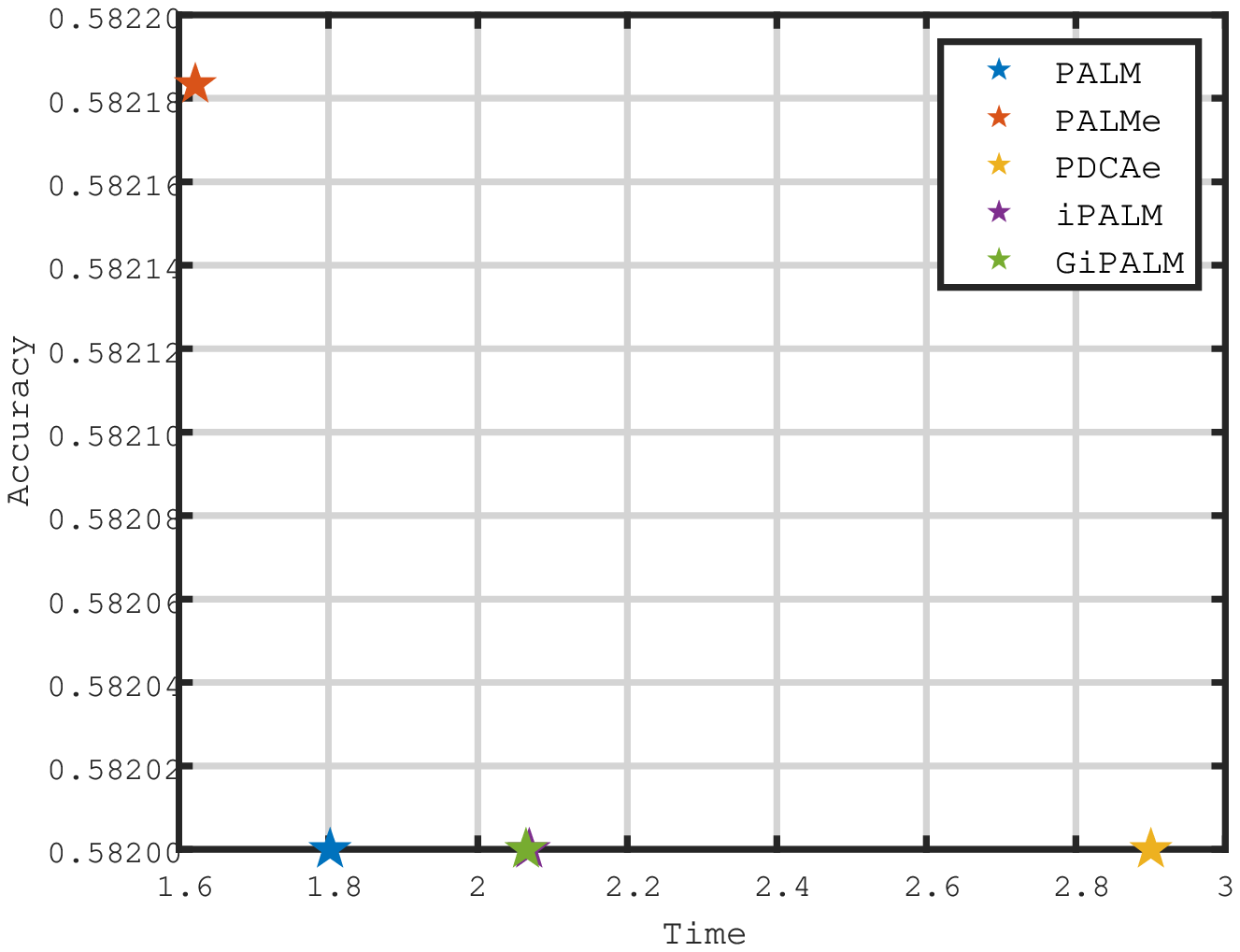}
		\caption{ $\emph{gisette}$}
	\end{subfigure}
	\begin{subfigure}[b]{0.22\textwidth}
		\centering
		\includegraphics[width=\textwidth]{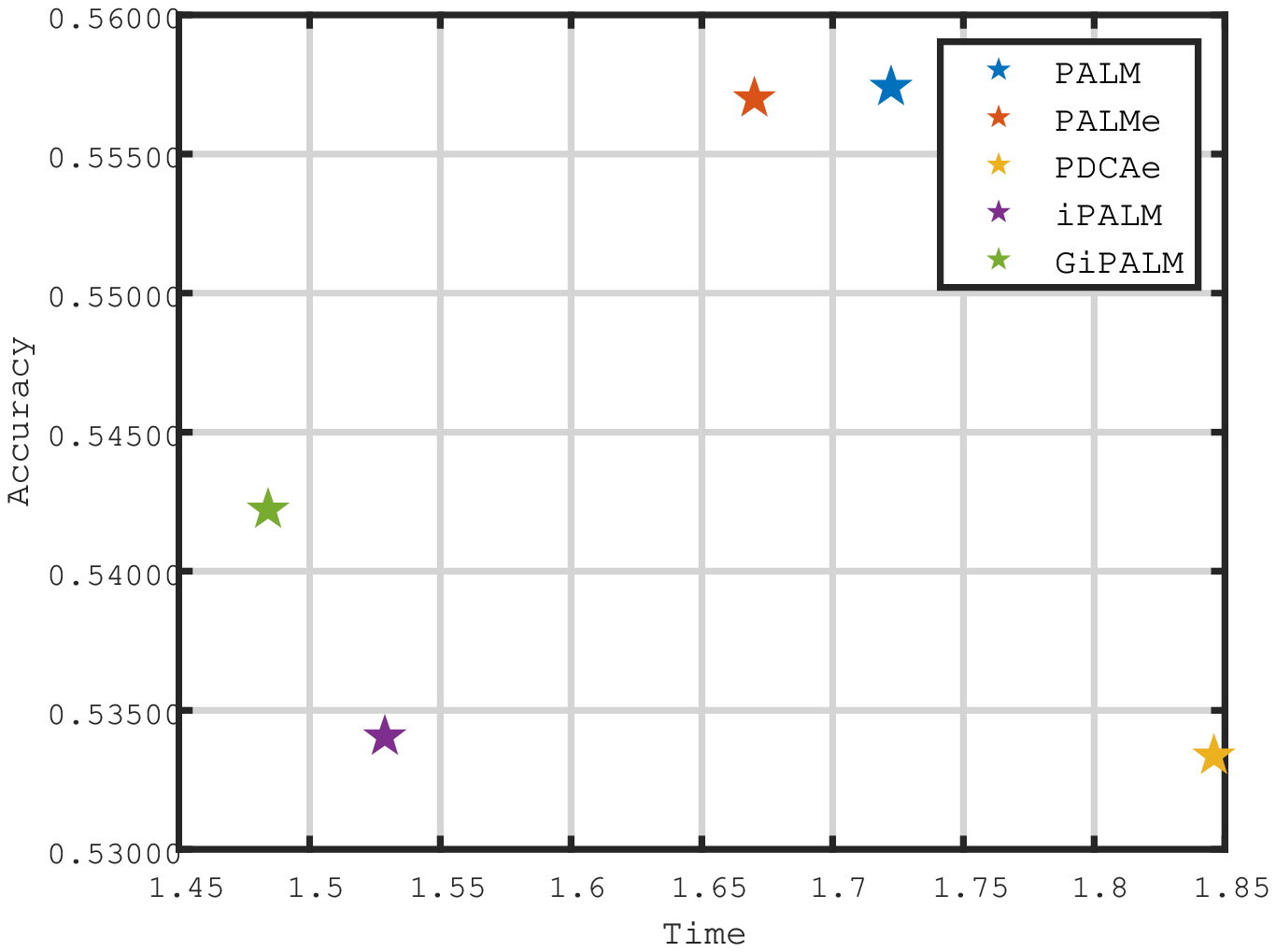}
		\caption{ $\emph{ijcnn}$}
	\end{subfigure}
	\caption{Clustering accuracy and running time on LIBSVM datasets}
	\label{fig:3}
\end{figure*}

\begin{table*}[!htbp]
	\caption{Dimension of datasets, dimension of subspace,  step-size and extrapolation parameters of the tested methods}
	\label{table-3}\vspace{-0.15in}
	\begin{center}
	\scalebox{0.93}{
		\begin{tabular}{c|ccccr}
			\hline
			\footnotesize $(\alpha,\beta,\gamma_P,\gamma_Q )$ & \footnotesize PALMe & \footnotesize PALM & \footnotesize iPALM & \footnotesize GiPALM  & \footnotesize pDCAe \\
			\hline
			\begin{tabular}{c}
				\footnotesize \emph{a6a} dataset \\
				\noalign{\vspace{-0.05cm}}
				\footnotesize $(n,d,K) = (11220,122,6)$
			\end{tabular} & \footnotesize $(10^{-7},10^{5},0,1)$ & \footnotesize $(10^{-8},10^{5},0,0)$ & \footnotesize $(10^{-7},10^{5},0.2,0.3)$ & \footnotesize $(10^{-6},5*10^{4},0.5,0.25)$ & \footnotesize $(-,5*10^{-6},-,1)$  \\ 
			\begin{tabular}{c}
				\footnotesize \emph{a9a} dataset \\
				\noalign{\vspace{-0.05cm}}
				\footnotesize $(n,d,K) = (32561,123,6)$
			\end{tabular} & \footnotesize $(10^{-8},10^{6},0,0.5)$ & \footnotesize $(10^{-7},10^{6},0,0)$ & \footnotesize $(10^{-8},10^{6},0.2,0.2)$ & \footnotesize $(10^{-8},10^{6},0.3,0.4)$ & \footnotesize $(-,10^{-6},-,0.3)$  \\
			\begin{tabular}{c}
				\footnotesize \emph{gisette} dataset \\
				\noalign{\vspace{-0.05cm}}
				\footnotesize $(n,d,K) = (6000,5000,1)$
			\end{tabular} & \footnotesize $(10^{-5},5000,0,1)$ & \footnotesize $(10^{-6},1000,0,0)$ & \footnotesize $(10^{-6},10^{4},0.05,0.05)$ & \footnotesize $(10^{-5},10^{4},0.05,0.05)$  & \footnotesize $(-,100,-,0.2)$   \\
			\begin{tabular}{c}
			\footnotesize \emph{ijcnn1} dataset \\
			\noalign{\vspace{-0.05cm}}
			\footnotesize $(n,d,K) = (49990,22,8)$
		    \end{tabular} & \footnotesize $(10^{-7},10^{4},0,0.1)$ & \footnotesize $(10^{-7},10^{4},0,0)$ & \footnotesize $(10^{-6},1.5*10^{4},0.05,0.05)$ & \footnotesize $(10^{-6},10^{4},0.2,0.2)$ & \footnotesize $(-,10^{-4},-,0.2)$  \\ 
			\hline
		\end{tabular}
	}
	\end{center}
\end{table*}

\subsection{Image Reconstruction}\label{subsec:image-re}

As suggested in \cite{markopoulos2014optimal,nie2021non}, the robustness to outliers can be visualized by image reconstruction. In this section, we conduct image reconstruction experiments on 6 gray-scale human face images (see the first row of Fig. \ref{fig:4}) downloaded from \emph{AT\&T} Database of Faces\footnote{\url{https://www.kaggle.com/datasets/kasikrit/att-database-of-faces?resource=download}} to test robustness of different models. We assume that the clean image $\bm{A}$ is not available and instead we have 9 corrupted versions $\bm{A}_1,\dots,\bA_9$. Here, these corrupted instances are generated by adding outliers to a block as follows. We firstly partition the original image $\bA$ into 9 equal-sized blocks, say $\bm{B}_1,\dots,\bm{B}_9$. Next, we select one block $\bm{B}_i$, divide it into 4 smaller blocks of equal size, add outliers drawn from the discrete uniform distribution in range $[1,200]$ to its diagonal blocks, and normalize it into range $[0,255]$ to generate $\bm{A}_i$ for $i=1,\dots,9$. One instance of the corrupted image is shown in the second row of Fig. \ref{fig:4}. Then, we can vectorize these 9 corrupted images by stacking them to form the data matrix $\bX$, i.e.,
\begin{align*}
\bm{X} = \begin{bmatrix}
\mathrm{vec}(\bA_1) & \dots & \mathrm{vec}(\bA_9).
\end{bmatrix}
\end{align*}
Now, we apply various PCA models to do image reconstruction on $\bX$, including our proposed approach, L1-PCA-based approach in \cite{wang2021linear}, R1-PCA-based approach in \cite{ding2006r}, matrix factorization based L1-PCA (MF-L1PCA) approach in \cite{yu2012efficient}, and L12-PCA-based approach in \cite{nie2021non}. We simply denote our approach by RI-L1PCA. The projection of $\bm{X}$ onto the learned subspace is referred as the reconstructed image, i.e, $\hat{\bX} = \bQ\bQ^T \bX$, where $\bQ$ is the returned subspace of the tested approach. 

In the experiments, the maximal iteration is set as 1000 and $K = 2$. For each algorithm,  we stop it when the Frobenious norm of the difference between two consecutive iterates is less then $10^{-3}$. Then, we can obtain the reconstructed images learned by different approaches in Fig. \ref{fig:4}. 
We can observe that the quality of the reconstructed images returned by our approach has comparable performance with others. 

%
\begin{figure}[!htbp]
	\centering
	\includegraphics[width=0.45\textwidth]{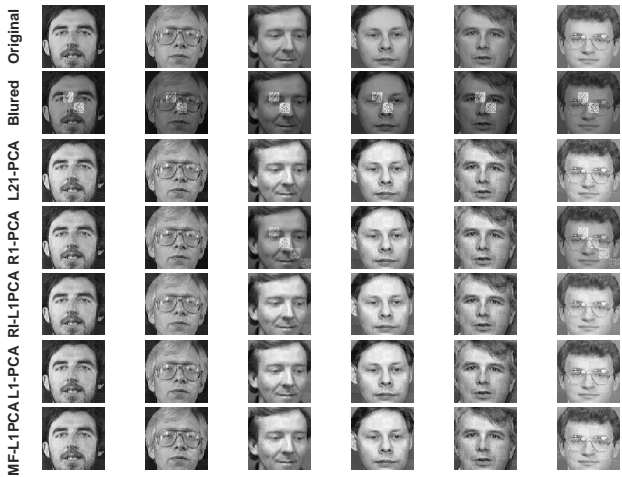}
	\caption{Image reconstruction of different models}
	\label{fig:4}
\end{figure}

\section{Conclusions}\label{sec:conc}
In this paper, we proposed a proximal alternating linearzed minimization method with a non-linear extrapolation to solve the rotationally invariant L1-PCA problem. To the best of our knowledge, this is the first work that presents a fast iterative method to tackle this non-smooth non-convex problem. 
We proved that our method converges at least linearly to a critical point of the L1-PCA problem under a mild condition. 
To demonstrate the efficacy of our method, extensive experimental results on both synthetic and real datasets are reported. In particular, we compare our approach with some existing robust PCA approaches to elucidate its robustness to outliers. Motivated by our proposed non-linear extrapolation, one future direction is to extend such extrapolation scheme to more general function class. Another possible direction is to design other optimization algorithms to tackle the rotationally invariant L1-PCA problem and analyze its convergence behavior. 

\bibliographystyle{IEEEtran}
\bibliography{ref}
\appendix

\subsection{Proof of Lemma \ref{lem:suffi-decre}} \label{first-appendix}
\begin{proof}
	We first prove (i). Due to $\bC^k \in {\mB}(n,d) \times {\rm St}(d,K) \times {\rm St}(d,K) $ for all $k \geq 0$ and the boundness of ${\mB}(n,d)$ and ${\rm St}(d,K)$, the sequence $\{\bC^k\}_{k\geq 0}$ is automatically bounded. 
	
	We next prove (ii). To simplify the notation, let
	\begin{equation*}
		\begin{split}
			&\ \bm{\Delta}_{\bP}^{k+1}=\bP^{k+1}-\bP^k,\ \bm{\Delta}_{\bQ}^{k+1}=\bQ^{k+1}-\bQ^k,\\
			&\ L_k = \|\bX\bP^k + \bP^{k^T}\bX^T\|,\ \forall\ k \ge 0.
		\end{split}
	\end{equation*}
	According to the updates \eqref{update:P-extra}, we have
	\begin{equation}\label{eq:pro}
		\langle\bX (\bP^{k+1}-\bP^k),\bE^k \rangle\geq \frac{\alpha_k}{2} \|\bP^{k+1}-\bP^k\|^2_F. 
	\end{equation}
	This, together with $H(\bP,\bQ) = -\langle \bP,\bX^T\bQ\bQ^T \rangle $ and $\bE^k =\bQ^{k}\bQ^{k^T}+\gamma_k(\bQ^{k}\bQ^{k^T}-\bQ^{k-1}\bQ^{{k-1}^T})$, implies
	\begin{equation*}\label{eq1:lem:suffi-decre}
		\begin{split}
			&\ H(\bP^{k+1},\bQ^k) - H(\bP^{k},\bQ^k) \le -\frac{\alpha_k}{2}\|\bm{\Delta}_{\bP}^{k+1}\|_F^2 \\
			&\ + \gamma_k \langle \bX(\bP^{k+1}-\bP^{k}), \bQ^k\bQ^{k^T} - \bQ^{k-1}\bQ^{{k-1}^T} \rangle \\
			&\ \le -\frac{\alpha_k}{2}\left( 1-\gamma_k \right)\|\bm{\Delta}_{\bP}^{k+1}\|_F^2  + \frac{2\gamma_k\|\bX\|^2}{\alpha_k}\|\bm{\Delta}_{\bQ}^k\|^2_F,
		\end{split}
	\end{equation*}
	where the second inequality is due to the following result:
	\begin{equation*}
		\begin{split}
			&\ \langle\bX(\bP^{k+1}-\bP^k),\bQ^k\bQ^{k^T}-\bQ^{k-1}\bQ^{{k-1}^T} \rangle\\
			= &\ \langle\bP^{k+1}-\bP^k,\bX^T\bQ^{k-1}(\bQ^{k^T}-\bQ^{{k-1}^T})\\
			+&\ \bX^T(\bQ^k-\bQ^{k-1})\bQ^{k^T} \rangle \le \frac{\alpha_k}{4}\|\bm{\Delta}_{\bP}^{k+1}\|_F^2\\
			+&\ \frac{\|\bX\bQ^{k-1}\|^2}{\alpha_k}\|\bm{\Delta}_{\bQ}^k\|^2_F
			+\frac{\alpha_k}{4}\|\bm{\Delta}_{\bP}^{k+1}\|_F^2+\frac{\|\bX\bQ^k\|^2}{\alpha_k}\|\bm{\Delta}_{\bQ}^k\|^2_F\\
			\le &\ \frac{\alpha_k}{2}\|\bm{\Delta}_{\bP}^{k+1}\|_F^2+\frac{2\|\bX\|^2}{\alpha_k}\|\bm{\Delta}_{\bQ}^k\|^2_F,
		\end{split}
	\end{equation*}
	where the first inequality uses the fact that $2\langle \bA,\bB \rangle \le \rho\|\bA\|_F^2 + \|\bB\|_F^2/\rho$ for any $\rho > 0$ and $\|\bA\bB\|_F \le \|\bA\|_F\|\bB\|_F$. The last inequality is due to $\|\bX\bQ\|_F \le \|\bX\|_F$ for any $\bQ \in {\rm St}(d,K)$. Moreover, according to the update \eqref{update:Q} and \cite[Lemma 2]{bolte2014proximal}, we have
	\begin{equation}\label{eq2:lem:suffi-decre}
		H(\bP^{k+1},\bQ^{k+1}) - H(\bP^{k+1},\bQ^k) \le -\frac{1}{2}\left( \beta_k - L_k \right)\|\bm{\Delta}_{\bQ}^{k+1}\|_F^2. 
	\end{equation}
	Then, we have
	\begin{equation*}
		\begin{split}
			&\ \Phi_{\beta_*}(\bC^{k+1}) - \Phi_{\beta_*}(\bC^k)\\
			=&\  H(\bP^{k+1},\bQ^{k+1}) - H(\bP^{k},\bQ^k)
			+  \frac{\beta_*}{2}\|\bm{\Delta}_{\bQ}^{k+1}\|^2_F - \frac{\beta_*}{2}\|\bm{\Delta}_{\bQ}^k\|^2_F \\
			\le&\ -\frac{\alpha_k}{2}\left( 1-\gamma_k \right)\|\bm{\Delta}_{\bP}^{k+1}\|_F^2 -\frac{1}{2}\left( \beta_k - L_k - \beta_* \right)\|\bm{\Delta}_{\bQ}^{k+1}\|_F^2 \\
			&\ -\frac{1}{2}\left( \beta_* - \frac{4\gamma_k\|\bX\|^2}{\alpha_k} \right)\|\bm{\Delta}_{\bQ}^k\|^2_F  \le -\kappa_1 \|\bC^{k+1} - \bC^k\|_F^2, 
		\end{split}
	\end{equation*}
	where the first inequality follow from \eqref{eq1:lem:suffi-decre} and \eqref{eq2:lem:suffi-decre} and the second inequality is due to \eqref{pararms} and $\kappa_1 = \min\{ \alpha_*(1-\gamma^*)/2,\beta^*/4 \}$. 

	Finally, we prove (iii). According to the updates \eqref{update:P-extra}, \eqref{update:Q}, we have
	\begin{equation} \label{eq1:opti}
		\z \in -\bX^T\bE^k +\alpha_k (\bP^{k+1}-\bP^k)+ \mN_{\mB(n,d)}(\bP^{k+1})
	\end{equation}
	\begin{equation}\label{eq2:opti}
		\begin{aligned}
			\z \in &\ -(\bX\bP^{k+1}+\bP^{{k+1}^T}\bX^T)\bQ^k\bQ^{k^T}\\
			&\ +\beta_k(\bQ^{k+1}-\bQ^k)+\mN_{{\rm St}(d,K)}(\bQ^{k+1}).
		\end{aligned}
	\end{equation}			
	This, together with \cite[Proposition 2.1]{attouch2010proximal}, yields
	\begin{equation}\label{parti:phi}
		\begin{aligned}
			&\ \partial \Phi_{\beta_*}\left(\bP^{k+1},\bQ^{k+1},\bQ^k\right)\\
			=&\ \left\{-\bX^T\bQ^{k+1}\bQ^{{k+1}^T} + \mN_{\mB(n,d)}(\bP^{k+1}) \right\}\\
			\times&\ \left\{\begin{split}
				&\  -(\bX\bP^{k+1}+\bP^{{k+1}^T} \bX^T) \bQ^{k+1} \bQ^{{k+1}^T} \\
				&\ + \beta_* (\bQ^{k+1}-\bQ^k )+\mN_{{\rm St}(d,K)}(\bQ^{k+1})
			\end{split} \right\}\\
			\times& \left\{ \beta_* (\bQ^k-\bQ^{k+1} )\right\}.
		\end{aligned}
	\end{equation}
	Combining this with \eqref{eq1:opti} and \eqref{eq2:opti}  yields
	\begin{equation*}
		\begin{split}
			&\ \dist^2\left(\z, \partial \Phi_{\beta_*}\left(\bP^{k+1},\bQ^{k+1},\bQ^k\right)\right)\\
			\le  &\ \|\bX^T(\bE^k-\bQ^{k+1}\bQ^{{k+1}^T})+\alpha_k\bm{\Delta}_{\bP}^{k+1}\|_F^2 \\
			+ &\ \|(\bX\bP^{k+1}+\bP^{{k+1}^T} \bX^T)(\bQ^{k} \bQ^{{k}^T}-\bQ^{k+1} \bQ^{{k+1}^T}) + \\
			&\ (\beta_*-\beta_k)\bm{\Delta}_{\bQ}^{k+1} \|_F^2 + \|\beta_*\bm{\Delta}_{\bQ}^{k+1}\|_F^2 \\
			\le &\ 4\|\bX\|^2(\|\bQ^{k+1}\bQ^{{k+1}^T}-\bQ^{k}\bQ^{{k}^T}\|_F^2 \\
			+ &\ \gamma_k^2\|\bQ^{k}\bQ^{{k}^T}-\bQ^{k-1}\bQ^{{k-1}^T}\|_F^2) + 2\alpha_k^2 \|\bm{\Delta}_{\bP}^{k+1}\|_F^2 \\
			+ &\ 2\|\bX\bP^{k+1}+\bP^{{k+1}^T} \bX^T\|^2\|\bQ^{k}\bQ^{{k}^T}-\bQ^{k-1}\bQ^{{k-1}^T}\|_F^2\\
			+ &\ 2(\beta_k-\beta_*)^2\|\bm{\Delta}_{\bQ}^{k+1} \|_F^2 + \beta_*^2\|\bm{\Delta}_{\bQ}^{k+1}\|_F^2 \\
			\le &\ 16\|\bX\|^2(\|\bm{\Delta}^{k+1}_{\bQ}\|_F^2 + \gamma^2_k\|\bm{\Delta}^{k}_{\bQ}\|_F^2) +  2\alpha_k^2 \|\bm{\Delta}_{\bP}^{k+1}\|_F^2\\
			+ &\ 8nd\|\bX\|^2\|\bm{\Delta}^{k}_{\bQ}\|_F^2 + (2(\beta_k-\beta_*)^2+\beta_*^2)\|\bm{\Delta}_{\bQ}^{k+1} \|_F^2 \\
			\le &\ 2\alpha_k^2 \|\bm{\Delta}_{\bP}^{k+1}\|_F^2 + 8(2\gamma^2_k+nd)\|\bX\|^2\|\bm{\Delta}^{k}_{\bQ}\|_F^2 \\
			+ &\ (16\|\bX\|^2+2(\beta_k-\beta_*)^2+\beta_*^2)\|\bm{\Delta}_{\bQ}^{k+1} \|_F^2.  
		\end{split} 
	\end{equation*}
	By taking $\kappa_2 = \max\{\sqrt{2}\alpha^*, 2\sqrt{4\gamma^{*^2}+2nd}\|\bX\|,\\ \sqrt{16\|\bX\|^2+2(\beta_k-\beta_*)^2+\beta_*^2}\}$, we obtain the desired result \eqref{eq:rela-err}. 
\end{proof}

\subsection{Proof of Corollary \ref{coro:KL}}\label{appendix:coro-KL}
\begin{proof} 
	Using Fact \ref{fact:QPOC} with setting $\bB=\bI$ and $\bA=(\bX\bP_i+\bP_i^T \bX^T)/{2}$, we have that the K\L\ exponent of $\ell_i\left(Q\right)$ is $1/2$ for all $i=1,\dots,2^{nd}$. 
Using the same argument in the proof in \cite[Section 3.3]{wang2021linear} yields that the K\L\ exponent of $h$ is $1/2$. This, together with \eqref{eq:potential} and \cite[Theorem 3.6]{li2018calculus}, implies the K\L\ exponent of $\Phi_{\beta}$ is $1/2$. 
\end{proof}

\subsection{Proof of Lemma \ref{relat:h-l}} \label{second-appendix}
\begin{proof} 
	We first prove (i). According to \cite[Proposition 2.1]{attouch2010proximal}, we have 
	\begin{equation}\label{parti:h}
		\begin{split}
			\partial h(\bP,\bQ) =&\ \left\{-\bX^T\bQ\bQ^{T} + \mN_{\mB(n,d)}(\bP) \right\} \\
			 \times &\ \{ -(\bX\bP+\bP^T\bX^T)\bQ +\mN_{{\rm St}(d,K)}(\bQ) \}. 
		\end{split}
	\end{equation}
	This, together with $(\z,\z) \in \partial h(\bP,\bQ)$, $\bP \in \sign(\bX^T\bQ\bQ^T)$, and \eqref{crit:l}, implies  
	the desired result. 
	Next, we prove (ii). According to $\z \in \partial \Phi_{\beta}(\bP,\bQ,\bQ^\prime)$ and  \eqref{parti:phi}, we obtain
	$\bQ^\prime = \bQ$. This, together with \eqref{parti:h}, implies that $\z \in \partial \Phi_{\beta}(\bP,\bQ,\bQ)$ if and only if $\z \in \partial h(\bP,\bQ)$. 
\end{proof}
\end{document}